\documentclass[11pt]{article}
\usepackage[numbers,sort&compress]{natbib}
\usepackage{enumerate}
\usepackage{amscd}
\usepackage{amsmath}
\usepackage{latexsym}
\usepackage{amsfonts}
\usepackage{amssymb}
\usepackage{amsthm}
\usepackage{verbatim}
\usepackage{mathrsfs}
\usepackage{enumerate}
\usepackage{hyperref}

\theoremstyle{plain}\newtheorem{definition}{Definition}[section]
\theoremstyle{definition}\newtheorem{theorem}{Theorem}[section]
\theoremstyle{plain}\newtheorem{lemma}[theorem]{Lemma}
\theoremstyle{plain}\newtheorem{coro}[theorem]{Corollary}
\theoremstyle{plain}\newtheorem{prop}[theorem]{Proposition}
\theoremstyle{remark}\newtheorem{remark}{Remark}[section]
\usepackage{xcolor}

\newcommand{\norm}[1]{\left\|#1\right\|}
\newcommand{\Div}{\mathrm{div}\,}

\newcommand{\be}{\begin{equation}}
\newcommand{\ee}{\end{equation}}
 \newcommand{\ba}{\begin{aligned}}
 \newcommand{\ea}{\end{aligned}}

  \newcommand{\f}{\frac}
    
  \newcommand{\ben}{\begin{enumerate}}
   \newcommand{\een}{\end{enumerate}}

\newcommand{\ti}{\nabla}

\newcommand{\Rmnum}[1]{\expandafter\@slowromancap\romannumeral #1@}

\allowdisplaybreaks

\numberwithin{equation}{section}
\begin{document}
\title{Analytical validation of   the helicity conservation for the  compressible  Euler equations }
\author{Yanqing Wang\footnote{   College of Mathematics and   Information Science, Zhengzhou University of Light Industry, Zhengzhou, Henan  450002,  P. R. China Email: wangyanqing20056@gmail.com},  ~  
      \, Wei Wei\footnote{School of Mathematics and Center for Nonlinear Studies, Northwest University, Xi'an, Shaanxi 710127,  P. R. China  Email: ww5998198@126.com }      ~  and   ~\, Yulin Ye\footnote{Corresponding author. School of Mathematics and Statistics,
      	Henan University,
      	Kaifeng, 475004,
      	P. R. China. Email: ylye@vip.henu.edu.cn} }
\date{}
\maketitle
\begin{abstract}
In \cite{[Moffatt]},   Moffatt introduced the concept of
 helicity  in   an inviscid fluid and    examined the helicity preservation of smooth solution to  barotropic compressible  flow. In this paper,
it is shown that the weak solutions of the above system in Onsager type spaces $\dot{B}^{1/3}_{p,c(\mathbb{N})}$ guarantee   the conservation of the helicity.
The  parallel results of  homogeneous incompressible  Euler equations and the surface quasi-geostrophic equation are also obtained.

  \end{abstract}
\section{Introduction}
\label{intro}
\setcounter{section}{1}\setcounter{equation}{0}
The
Euler equations describing  incompressible inviscid fluids with constant density in $\Omega\times(0,\,T)$ read
\begin{equation}\left\{\begin{aligned}\label{Euler}
&v_{t}+v\cdot \nabla v+\nabla \Pi=0,\\ &\text{div}\, v=0,\\
&v|_{t=0}=v_{0}(x),
\end{aligned}\right.\end{equation}
 where the unknown vector $v $ represents   the flow  velocity field, and the scalar function $\Pi=-\Delta^{-1}\partial_{i}\partial_{j}(v_{i}v_{j})$ stands for  the   pressure. The
initial  velocity $v_0$ satisfies   $\text{div}\,v_0=0$. Let $\Omega$ be the whole space $\mathbb{R}^{d}$ or the periodic domain $\mathbb{T}^{d}$, where $d\geq 2$ is the spatial dimension. We denote $\omega=\text{curl\,}v$ by   the vorticity of the   flow  velocity in \eqref{Euler}, whose equations read
\be
\omega_{t}+v\cdot\nabla\omega-\omega\cdot\nabla v=0,~~\text{div\,}\omega=0.\label{vorticityeq}
\ee
It is well known that there exist two physically conserved  quantities for the regular solutions of the Euler system \eqref{Euler}:
 \begin{itemize}
\item  energy  conservation
\be\label{ec}
\int_{\Omega}|v(x,t)|^{2}dx=\int_{\Omega}|v(x,0)|^{2}dx.
\ee
\item helicity preservation
\be\label{hc}
\int_{\Omega} \omega(x,t)\cdot v(x,t)  dx=\int_{\Omega} \omega(x,0)\cdot v(x,0)  dx.
\ee
 \end{itemize}
 Besides the above two conserved laws  of the Euler equations \eqref{Euler}, one can find others in \cite[p.24]{[MB]}.
In a  seminal paper \cite{[Onsager]},
  Onsager conjectured that weak solutions of the Euler equation with
  H\"older continuity exponent $\alpha>\frac{1}{3}$ do conserve energy and that turbulent
  or anomalous dissipation occurs when $\alpha\leq \frac{1}{3}$. The first attempt to the energy conservation issue for weak solutions of the  Euler equations  \eqref{Euler} was given by Eyink \cite{[Eyink]} in a stronger H\"older space $C_*^\alpha$ with $\alpha >\frac{1}{3}$,
  which is included in $C^\alpha$. Subsequently Constantin-E-Titi \cite{[CET]} successfully solved  the positive part of the
   Onsager's conjecture  and proved that the  energy    is  conserved if a weak solution $v$ is   in the Besov space $L^{3}(0,T; B^{\alpha}_{3,\infty}(\mathbb{T}^{3}))$ with $\alpha>1/3$. Later on,  by deriving a local energy equation which contains a term $D(v)$ representing the dissipation or production of energy caused by the lack of smoothness of solution, Duchon and Robert gave a weaker condition on the solutions to conserve energy in  \cite{[DR]}, that is, if $v$ satisfies $\int|v(t,x+\xi)-v(t,x)|^3dx \leq C(t)|\xi|\sigma(|\xi|)$, where $C(t)\in L^1(0,T)$ and $\sigma(a)\rightarrow 0$ as $a\rightarrow 0$, then energy is conserved. Recently,
 Cheskidov-Constantin-Friedlander-Shvydkoy \cite{[CCFS]}  refined the above spaces  to the
   critical space $L^{3}(0,T; B^{1/3}_{3,c(\mathbb{N})})$,  where
$
B^{1/3}_{3,c(\mathbb{N})}=\{v\in B^{1/3}_{3,\infty}: \lim_{q\rightarrow\infty}2^{q}\|\Delta_{q}v\|^{3}_{L^{3}}=0\}$ and  $\Delta_{q}$ represents a smooth restriction of $v$ into Fourier modes of order $2^q$. In the follow-up paper, Shvydkoy \cite{[Shvydkoy2009],[Shvydkoy2010]} stated that this condition of energy conservation can be equivalent to
$$
\lim_{z\rightarrow0}\f{\int_{0}^{T}\int_{\mathbb{R}^d}|v(x+z,t)-v(x,t)|^{3}dxdt}{|z|}=0.
$$ It is worth remarking that $B^{1/3}_{3,q}$ for $1\leq q<\infty$ are included in $B^{1/3}_{3,c(\mathbb{N})}$, which is called as the Onsager-critical space.
  In this direction, the recent progress  can be found in
 \cite{[FW2018]}.
  The proof of the other part of
   Onsager's conjecture
   in three-dimensional space is given by Isett \cite{[Isett]}, which adopted the approach developed
   by De Lellis and  Szekelyhidi \cite{[DS0],[DS1],[DS2],[DS3]}.

  Next, we turn our attention to the helicity conservation \eqref{hc} in the Euler equations.
The concept of the
 helicity  in an inviscid fluid was introduced by Moffatt in \cite{[Moffatt]}. As pointed in \cite{[Moffatt],[MT]},
helicity is important at a fundamental level in relation to flow kinematics
because it admits topological interpretation in relation to the linkage or
linkages of vortex lines of the flow.
Indeed,
Moffatt \cite{[Moffatt]} examined that the    helicity of smooth solutions to
the  following compressible Euler equations  \eqref{CEuler} is invariant in time,
\be\left\{\ba\label{CEuler}
&\rho_t+\nabla \cdot (\rho v)=0, \\
&(\rho v)_{t} +\Div(\rho v\otimes v)+\nabla
\pi(\rho )=0.
\ea\right.\ee
The proof   in \cite{[Moffatt]} relies on the transport formula involving moving region. To the best of our knowledge, there has been little literature  concerning the  helicity preservation  of weak solutions to
the    compressible Euler equations \eqref{CEuler}. In  this direction,  most previous works     focus  on
 the homogeneous incompressible Euler equations \eqref{Euler}.  In particular,
the study of helicity conservation for the incompressible Euler equations \eqref{Euler}
was originated from
Chae's interesting works   \cite{[Chae1],[Chae]}.
 In \cite{[Chae]}, Chae considered the preservation of the helicity  of weak solutions to the Euler equations  via $\omega$
in spaces $ L^{3}(0,T;B^{\alpha}_{\f95,\infty}) $ with $\alpha>\f13$.
Subsequently,
it is shown that $v \in L^{r_{1}}(0,T;\dot{B}^{\alpha}_{\f92,q})$ and
$\omega \in  L^{r_{2}}(0,T;\in \dot{B}^{\alpha}_{\f95,q})$, with $\alpha>\f13$, $q\in [2,\infty]$, $r_1\in [2,\infty]$, $r_2\in [1,\infty]$ and $\frac{2}{r_{1}}+\frac{1}{r_{2}}=1$,  ensure  the helicity conservation of weak solutions in \cite{[Chae1]}.  After that,  Cheskidov-Constantin-Friedlander-Shvydkoy \cite{[CCFS]} established the helicity conservation class based on the velocity $v$ in  Onsager critical space
\be\label{ccfs}
 L^{3}(0,T;B^{\f23}_{3,c(\mathbb{N})}).
\ee
 Very recently, De Rosa \cite{[De Rosa]} proved that the helicity is a constant provided that
$v\in L^{2r}(0,T;W^{\theta, 2p})$ and  $ \omega\in L^{\kappa}(0,T;W^{\alpha, q})$ with  $\f{1}{p}+\f1q=\f{1}{r}+\f{1}{\kappa}$  and $2\theta+\alpha\geq1$.
Notice that there holds the embedding relation $W^{\theta,p}\hookrightarrow B^{\theta}_{p,c(\mathbb{N})}$, which  can be found in  \cite{[De Rosa],[CCFS]}.
  The first objective of this paper is to address  the question
  how much regularity is needed for a weak solution  of the compressible Euler equations  \eqref{CEuler}  to conserve the helicity.
To this end,    we shall need to give the equation for the vorticity $\omega$ from \eqref{CEuler}. Our observation is that for any smooth solutions $(\rho, v)$ of compressible Euler equations \eqref{CEuler} with $0<c_1\leq \rho\leq c_2<\infty$, dividing the both sides of the momentum equation $\eqref{CEuler}_2$ by $\rho$, we can reformulate the system    \eqref{CEuler} as
\be\left\{\ba\label{rCEuler1}
&\rho_t+\nabla \cdot (\rho v)=0, \\
& v_{t} + v\cdot\nabla v+\nabla
\Pi(\rho )=0,
\ea\right.\ee
where  $\Pi(\rho )= \int_{1}^{\rho}\f{\pi'(s)}{s}ds$.
 Then the momentum equation  in \eqref{rCEuler1} is the same as that in
the incompressible case.  However, this is not valid for weak solutions $(\rho,v)$ even with $0<c_1\leq \rho\leq c_2<\infty$. To use the equations \eqref{rCEuler1},  we shall apply the techniques due to
Feireisl,   Gwiazda,   Swierczewska-Gwiazda and   Wiedemann in   \cite{[EGSW]} to show that the weak solutions of    compressible Euler equations  \eqref{CEuler} in Onsager type spaces $B^{1/3}_{p,c(\mathbb{N})}$  are also the weak solutions of equations \eqref{rCEuler1}.
Now we state our first main result on the helicity conservation of weak solutions for the  compressible Euler equations \eqref{rCEuler1} as follows.
\begin{theorem}\label{the01}
	Let $(\rho,v)$ be a solution of \eqref{CEuler} in the sense of distributions and $\text{div\,}v,\text{curl\,}v\in C([0,T];L^{\f{2d}{d+1}}(\Omega))$. Assume that there exist two positive constants $c_1$ and $c_2$ such that
	\begin{equation}\label{a1}\ba
	&0<c_1\leq \rho \leq c_2<\infty,~	~\rho \in L^3(0,T;\dot{B}^{\frac{1}{3}}_{3,c(\mathbb{N})}),~\rho v \in L^3(0,T;\dot{B}^{\frac{1}{3}}_{3,\infty}),\\
&v \in L^3(0,T;B^{\frac{1}{3}}_{3,c(\mathbb{N})}),
  \pi\in C^2[c_1,c_2],
  \ea\end{equation}
  and one of the following four conditions is satisfied
 \begin{enumerate}[(1)]
 	\item$\omega\in L^{3}(0,T;\dot{B}^{\frac{1}{3}}_{3,\infty});$
 	 \item $\omega \in L^3(0,T;L^3(\Omega));$
\item $v\in L^{\f{p}{p-2}}(0,T;L^{\f{q}{q-2}}(\Omega)), \omega\in L^{p}(0,T;L^{q}(\Omega)),$ with $2< p,q<\infty$;
\item $\text{div\,}v,\omega\in L^{3}(0,T;L^{\f{9}{4}}(\Omega)),$ with $d=3$.
    \end{enumerate}
	Then the helicity of the compressible Euler equation \eqref{CEuler} is conserved, that is, for any $0<t<T$,
$$	
\int_{\Omega} \omega(x,t)\cdot v(x,t)  dx=\int_{\Omega} \omega(x,0)\cdot v(x,0)  dx.
$$	
	\end{theorem}

 \begin{remark}
  The   helicity conservation of weak solutions for the compressible Euler equations  \eqref{CEuler}
in Onsager's critical space $B^{\theta}_{p,c(\mathbb{N})}$ and $B^{\theta}_{p,\infty}$ in terms of the velocity and the vorticity  is studied in
  Theorem \ref{the1.1} and   Corollary  \ref{coro1}. This is  consistent with Moffatt's work \cite{[Moffatt]}.  The regularity condition that $\text{div\,}v,\text{curl\,}v\in C([0,T];L^{\f{2d}{d+1}}(\Omega))$ is necessary to ensure that the  helicity is well defined.
\end{remark}
\begin{remark}
The constraint that $\rho v\in L^3(0,T;\dot{B}^{\frac{1}{3}}_{3,\infty})$ can be removed, if the condition that $\rho \in L^3(0,T;\dot{B}^{\frac{1}{3}}_{3,c(\mathbb{N})}) $ is replaced by $\rho \in L^\infty(0,T;\dot{B}^{\frac{1}{3}}_{\infty,c(\mathbb{N})})$ due to Lemma \ref{lem2.5}.
\end{remark}
 \begin{remark}
This theorem is valid for the compressible isentropic Euler system. Indeed, for the isentropic pressure law   $\pi(\rho)=\kappa\rho^{\gamma}$ with $\gamma>1$ and $\kappa=\f{(\gamma-1)^{2}}{4\gamma},$ the pressure condition that $\pi\in C^2[c_1,c_2]$ is obviously satisfied if  $0<c_1\leq \rho \leq c_2<\infty$.
 \end{remark}
  \begin{remark}
 The sufficient conditions (3) in this theorem and  (4)-(5) in  Theorem  \ref{the1.1} below are partially motivated by the recent works \cite{[WY]} on energy  balance of the Navier-Stokes equations based on the combination   of the velocity and its gradient.
\end{remark}
 \begin{remark} Since the following embedding relations in $\mathbb{R}^{3}$ is valid
 $$W^{1,\f94}\hookrightarrow H^{\f76}\hookrightarrow B^{\f23}_{3,c(\mathbb{N})},
 $$
 the fourth criterion of this theorem is also in Onsager's critical space.
 It is an interesting question to show
the velocity $v$ in  Onsager's critical space
$   L^{3}(0,T;B^{\f23}_{3,c(\mathbb{N})}) $  keeping the  helicity conservation for the compressible Euler equations.
\end{remark}
In what follows, we formulate our second main result involving helicity conservation for the
incompressible Euler equations \eqref{Euler}.
 \begin{theorem}\label{the1.1} Let $ v$ be a  weak solution of  incompressible Euler equations \eqref{Euler} in the sense of Definition \ref{eulerdefi}  and $\omega\in C([0,T];L^{\f{2d}{d+1}}(\Omega))$. Then the helicity  conservation \eqref{hc} is valid provided that one of the following five conditions is
satisfied
 \begin{enumerate}[(1)]
 \item $ v\in L^{k} (0,T;\dot{B}^{\alpha}_{p,c(\mathbb{N})} ), \omega\in L^{ \ell  } (0,T; \dot{B}^{\beta}_{q,\infty}),$ with $\f2k+\f1\ell=1,\f2p+\f1q =1,2\alpha+\beta\geq1;$
\item $ v\in L^{k} (0,T;\dot{B}^{\alpha}_{p,\infty} ), \omega\in L^{ \ell  } (0,T; \dot{B}^{\beta}_{q,c(\mathbb{N})}),$ with $\f2k+\f1\ell=1,\f2p+\f1q =1,2\alpha+\beta\geq1;$
 \item
    $\omega\in L^{3} (0,T; \dot{B}^{\f13}_{\f{3d}{d+2},c(\mathbb{N})});$

\item $v\in L^{\f{p}{p-2}}(0,T;L^{\f{q}{q-2}}(\Omega)), \omega\in L^{p}(0,T;L^{q}(\Omega)),$ with $2<p, q<\infty$;
          \item $v\in L^{\f{2p}{p-1}}(0,T;L^{\f{2q}{q-1}}(\Omega) ) , \text{curl}\,\omega\in L^{p}(0,T;L^{q}(\Omega)),$ with $1\leq p\leq\infty,1\leq q<\infty$.
     \end{enumerate}
\end{theorem}
\begin{remark}
Theorem \ref{the1.1}  is an improvement of corresponding results in \cite{[De Rosa],[Chae1],[Chae]}.
 \end{remark}
\begin{remark}
  An   approach to \eqref{ccfs} in  homogeneous Besov spaces is   presented in the proof of Theorem \ref{the1.1}.
\end{remark}
  Recently, when dimension $d=3$, the authors in \cite{[LWY]} provide a proof of    energy conservation criteria for the Euler equations
 in terms of  the vorticity
 \be\label{lwy}
   \omega\in L^{3}(0,T;L^{\f{9}{5}}(\Omega)),
 \ee
 which is also deduced from  Cheskidov-Constantin-Friedlander-Shvydkoy's  classical  condition that $L^{3}(0,T; B^{1/3}_{3,c(\mathbb{N})})$ (see also related work \cite{[CLNS]}).
In the spirit of \cite{[LWY]}, we have
\begin{coro}\label{coro1}
Let $ v$ be a  weak solution of the 3D  incompressible Euler equations \eqref{Euler} in the sense of Definition \ref{eulerdefi}  and $\omega\in C([0,T];L^{\f{3}{2}}(\Omega))$. Then the helicity  conservation \eqref{hc} is valid provided that one of the following two conditions is
satisfied
 \begin{enumerate}[(1)] \item      $  \omega\in L^{3}(0,T;L^{\f{9}{4}}(\Omega))$;
\item      $  \text{curl\,}\omega\in L^{3}(0,T;L^{\f{9}{7}}(\Omega))$.
     \end{enumerate}
\end{coro}
\begin{remark}
 From \eqref{lwy} and the condition (1) in Corollary \ref{coro1}, we see that there may exist   a weak solution of the Euler equations that keep  the   energy   rather than the helicity. This is also previously
pointed out by Chae in  \cite{[Chae]}.
 \end{remark}
\begin{remark}
 In dimension two, the criterion  (1) in Corollary \ref{coro1} can be replaced by a new condition that
$\omega\in L^{3}(0,T;L^{2^{+}}(\Omega))$.
\end{remark}
  The starting point
  of  (1) in Theorem \ref{the01}  and  (1)-(2) in Theorem  \ref{the1.1} is the study of the cross-helicity in the ideal MHD equations.
Precisely, for any vector functions $A,B$, there hold the identities below$$\ba
&\nabla(A\cdot B)=A\cdot\nabla B+B\cdot\nabla A+A\times\text{curl}B+B\times\text{curl}A,\\
&\text{curl}(A\times B)=A \text{div} B-B \text{div}  A+B\cdot\nabla A-A\cdot\nabla B,
\ea$$
which together with the condition that $ \Div \omega =0$ lead to
 \be\label{identity}\ba
&v\cdot\nabla v = \f{1}{2}\nabla |v|^{2}+\omega\times v,\\
&\text{curl}(\omega\times v)=\omega \text{div} v+ v\cdot\nabla \omega-\omega\cdot\nabla v.\ea\ee
Then we can use \eqref{identity} to  reformulate the compressible Euler system \eqref{rCEuler1}  and its  vorticity equations as
\begin{equation}\label{euler}\left\{\begin{aligned}
&v_{t} +\omega\times v+\nabla(\Pi(\rho)+\f12|v|^{2})= 0,\\
&\omega_{t}-\text{curl}( v\times\omega)=0,\\
& \text{div\,}\omega =0,
\end{aligned}\right.\end{equation}
which is very closed to the  inviscid ideal Magnetohydrodynamics (MHD) equations
$$\left\{\ba\label{iMHD}
&v_{t} +   \omega\times v -b\cdot\ti b +\nabla (p+\f12|v|^2+\f12 |b|^2)= 0,  \\
&b_{t} -\text{curl}(v\times b)= 0, \\
&\Div v=\Div b=0,
\ea\right.$$
where $ b$ describes  the magnetic field.
For this  ideal MHD system, the cross-helicity
$\int_{\Omega} v(x,t)\cdot h(x,t)dx$ is conserved for the smooth solutions. To the knowledge of the authors,  Yu \cite{[Yu]} obtained  cross-helicity conservation criterion of weak solutions for ideal MHD equations based on the condition that
$v\in L^{3}(0,T;B^{\alpha_{1}}_{3,\infty})$ and $b\in L^{3}(0,T;B^{\alpha_{2}}_{3, c(\mathbb{N})})$ with $\alpha_{1}+2\alpha_{2}\geq1$ and $\alpha_{2}\geq\f13$ (see also \cite{[KL],[WZ],[CKS]}).
Inspired by Yu's work \cite{[Yu]}, we consider  the   helicity conservation of weak solutions to the incompressible Euler equations \eqref{Euler} in spaces $B^{\theta}_{p,c(\mathbb{N})}$ and $B^{\theta}_{p,\infty}$. It should be pointed out that this coincides
with the appearance of helicity conservation of weak solutions to the compressible Euler equations \eqref{CEuler} motivated by  the   cross-helicity conservation in the MHD system (see \cite{[Moffatt],[MT]}). The proof of (1) in Theorem \ref{the01}  and  (1)-(2) in Theorem  \ref{the1.1} is very closed to that in \cite{[CCFS],[Yu]}.  The key point is our new observation that
functions  in the  Onsager type spaces $\dot{B}^{1/3}_{p,c(\mathbb{N})}$ can satisfy the   corresponding    Constantin-E-Titi type commutator estimates  in  physical spaces (see Lemma \ref{lem2.3}).
It seems that the argument  presented here is quite general and can be applied in other fluid equations.
Particularly, we shall show that our approach here can be applied in the  2D    surface quasi-geostrophic (SQG) equation  below,
\be\left\{\ba\label{qg}
&\theta_{t} + v\cdot\ti
\theta=0, ~~~~~~~~~~~~~~~~~~~~~~~~~~~~~~~~~~~~~~~\text{in } ~(0,T)\times \Omega,\\
&v=\nabla^\perp \Psi=(-\frac{\partial \Psi}{\partial_{x_2}},\frac{\partial \Psi}{\partial_{x_1}}),~\theta=-(-\Delta )^{\frac{1}{2}} \Psi, ~~~~~\text{in } ~[0,T)\times \Omega,\\
&\theta|_{t=0}=\theta_0,~~~~~~~~~~~~~~~~~~~~~~~~~~~~~~~~~~~~~~~~~~~~~~\text{in } ~\Omega,
\ea\right.\ee
where the unknown scalar function
$\theta(x, t)\colon \mathbb{R}^2\times[0,\infty)\to \mathbb{R}$ stands for  the temperature, $v$ is the velocity field, and  $\Psi=-\int_{\Omega} \frac{\theta(y,t)}{|x-y|}dy$ is the stream function. Here we assume that $\Omega$ is the whole space $\mathbb{R}^2$
or torus $\mathbb{T}^2$ with periodic boundary conditions. The surface quasi-geostrophic equation arises in geophysical fluids and shares many striking similarities with  the 3D Euler equations  (see \cite{[CMT]} and   references
therein).  Zhou  \cite{[Zhou]} studied 
the following 
  general helicity
$$
\int_{\Omega}  \theta \partial_{i}\theta   dx,i=1,2,
$$
 of weak solutions for the  2D surface quasi-geostrophic equation  \eqref{qg}. It is shown that if
$\nabla\theta \in C([0,T];L^{\f43})\cap L^{3}(0,T;B^{\alpha}_{\f32,\infty})$ for $\alpha>1/3$, then the general helicity of weak solutions for 2-D surface quasi-geostrophic equation is conserved in \cite{[Zhou]}.
 In the present paper, we are going to give some new sufficient conditions to guarantee the
conservation of the helicity    for weak solutions to the 2-D surface quasi-geostrophic equation  in Onsager-critical spaces. Our third main result can be stated as follows:
\begin{theorem}\label{the1.2} Let  $\theta$ be a  weak solution to the 2D surface quasi-geostrophic equation  \eqref{qg} in the sense of Definition \ref{qgdefi} and $ \nabla\theta\in C([0,T];L^{\f43}(\Omega))$. Then the helicity  conservation \be\label{ghqg}
\int_{\Omega}   \theta(x,t) \partial_{i}\theta(x,t)   dx =\int_{\Omega}  \theta_0(x) \partial_{i}\theta_0 (x) dx, i=1,2, \ee
 is valid provided that
  $ \nabla\theta \in L^{3} (0,T;\dot{B}^{\f13}_{\f32,c(\mathbb{N})} )\cap C([0,T];L^{\f43}(\Omega)). $
\end{theorem}
 \begin{remark}
According to the Bernstein inequality, one may replace the condition that $ \nabla\theta \in L^{3} (0,T;\dot{B}^{\f13}_{\f32,c(\mathbb{N})} ) $ by $  \theta \in L^{3} (0,T;\dot{B}^{\f43}_{\f32,c(\mathbb{N})} ) $ in this theorem.
\end{remark}
The rest of this paper is organized as follows.
 In Section 2, we present some notations and  auxiliary lemmas which will be used in the present paper. Particularly, in the spirit of \cite{[Yu]},  we shall  show that   the  functions in  Onsager type spaces $\dot{B}^{\f13}_{p,c(\mathbb{N})}$ mean the Constantin-E-Titi type commutator estimates  in  physical spaces, which plays an important role in the follow-up study.
 Section 3 and Section 4 are devoted to  the proof of helicity conservation of weak solutions to the compressible Euler equations and the homogeneous incompressible Euler equations, respectively.  The helicity conservation of weak solutions to the  2-D surface quasi-geostrophic equation  is considered  in  Section 5.
\section{Notations and some auxiliary lemmas} \label{section2}

First, we introduce some notations used in this paper.
 For $p\in [1,\,\infty]$, the notation $L^{p}(0,\,T;X)$ stands for the set of measurable functions $f$ on the interval $(0,\,T)$ with values in $X$ and $\|f\|_{X}$ belonging to $L^{p}(0,\,T)$. The classical Sobolev space $W^{k,p}(\Omega)$ is equipped with the norm $\|f\|_{W^{k,p}(\Omega)}=\sum\limits_{|\alpha| =0}^{k}\|D^{\alpha}f\|_{L^{p}(\Omega)}$.
  $\mathcal{S}$ represents the Schwartz class of rapidly decreasing functions, $\mathcal{S}'$ the
space of tempered distributions, $\mathcal{S}'/\mathcal{P}$ the quotient space of tempered distributions which modulo polynomials.
  We use $\mathcal{F}f$ or $\widehat{f}$ to denote the Fourier transform of a tempered distribution $f$. $a\approx b$ means that $C^{-1}b\leq a\leq Cb$ for some constant $C>1$. For simplicity, we write $$\int_0^t\int_{\Omega} f(t, x)dxds=\int_0^t\int f\ ~~\text{and}~~ \|f\|_{L^p(0,T;X(\Omega) )}=\|f\|_{L^p(0,T;X)}.$$
To define Besov  spaces, we need the following dyadic unity partition
(see e.g. \cite{[BCD]}). Choose two nonnegative radial
functions $\varrho$, $\varphi\in C^{\infty}(\mathbb{R}^{d})$
supported respectively in the ball $\{\xi\in
\mathbb{R}^{d}:|\xi|\leq \frac{3}{4} \}$ and the shell $\{\xi\in
\mathbb{R}^{n}: \frac{3}{4}\leq |\xi|\leq
  \frac{8}{3} \}$ such that
\begin{equation*}
 \varrho(\xi)+\sum_{j\geq 0}\varphi(2^{-j}\xi)=1, \quad
 \forall\xi\in\mathbb{R}^{d}; \qquad
 \sum_{j\in \mathbb{Z}}\varphi(2^{-j}\xi)=1, \quad \forall\xi\neq 0.
\end{equation*}
Then for every $\xi\in\mathbb{R}^{d},$ $\varphi(\xi)=\varrho(\xi/2)-\varrho(\xi)$. Denote $h=\mathcal{F}^{-1} \varphi $ and $\tilde{h}=\mathcal{F}^{-1}\varrho$, then nonhomogeneous dyadic blocks  $\Delta_{j}$ are defined by
$$
\Delta_{j} u:=0 ~~ \text{if} ~~ j \leq-2, ~~ \Delta_{-1} u:=\varrho(D) u =\int_{\mathbb{R}^d}\tilde{h}(y)u(x-y)dy,$$
$$\text{and}~~\Delta_{j} u:=\varphi\left(2^{-j} D\right) u=2^{jd}\int_{\mathbb{R}^d}h(2^{j}y)u(x-y)dy  ~~\text{if}~~ j \geq 0.
$$
The nonhomogeneous low-frequency cut-off operator $S_j$ is defined by
$$
S_{j}u:= \sum_{k\leq j-1}\Delta_{k}u.$$
The homogeneous dyadic blocks $\dot{\Delta}_{j}$ and homogeneous low-frequency cut-off operators $\dot{S}_j$ are  defined  for every $j\in\mathbb{Z}$ by
\begin{equation*}
  \dot{\Delta}_{j}u:= \varphi(2^{-j}D)u=2^{jd}\int_{\mathbb{R}^d}h(2^{j}y)u(x-y)dy,
\end{equation*}
$$ \text { and }~~ \dot{S}_{j}u:=\varrho(2^{-j}D)u=2^{jd}\int_{\mathbb{R}^d}\tilde{h}(2^{j}y)u(x-y)dy.$$
Then for $-\infty <s<\infty$ and $1\leq p,q\leq \infty,$ the homogeneous Besov semi-norm $ \|f\|_{\dot{B}^{s}_{p, q}}$ of $f\in \mathcal{S}'/\mathcal{P}$ is given by
\begin{equation*}
	\begin{aligned}
  \norm{f}_{\dot{B}^{s}_{p, q}}:=\left\{\begin{array}{lll}\left(\sum_{j\in \mathbb{Z}}2^{jqs}\norm{\dot{\Delta}_{j} f} _{L^p}^q\right)^{1/q},~~\text{if}\ q\in [1,\infty),\\
  	\sup_{j\in \mathbb{Z}}2^{js}\norm {\dot{\Delta}_{j}f} _{L^p},~~~~~~~~~\text{if}~q=\infty.\end{array}\right.
\end{aligned}\end{equation*}
Moreover, for $s\in\mathbb{R}$ and $1\leq p,q\leq \infty$, we define the inhomogeneous Besov norm $\norm{f}_{B^s_{p,q}}$ of $f\in \mathcal{S}^{'}$ as
$$\norm{f}_{B^s_{p,q}}=\norm{f}_{{L^p}}+\norm{f}_{\dot{B}^s_{p,q}}.$$
Motivated by \cite{[CCFS]}, we denote $\dot{B}^\alpha _{p,c(\mathbb{N})}$ with $\alpha\in\mathbb{R}$ and $1\leq p \leq \infty$ as the class of all tempered distributions $f$ satisfying
\begin{equation}\label{2.1}
\norm{f}_{\dot{B}^\alpha _{p,\infty}}<\infty~ \text{and}~ 	\lim_{j\rightarrow \infty} 2^{j\alpha}\norm{\dot{\Delta}_j f}_{L^p}=0.
\end{equation}
It is clear that the Besov
spaces $\dot{B}^\alpha_{p,q}$ are included in $\dot{B}^\alpha_{p,c(\mathbb{N})}$ for any $1\leq q< \infty$. For more background on harmonic analysis in the context of fluids, we refer the readers to \cite{chemin}.

Furthermore, we let $\eta_{\varepsilon}:\mathbb{R}^{d}\rightarrow \mathbb{R}$ be a standard mollifier, i.e. $\eta(x)=C_0e^{-\frac{1}{1-|x|^2}}$ for $|x|<1$ and $\eta(x)=0$ for $|x|\geq 1$, where $C_0$ is a constant such that $\int_{\mathbb{R}^d}\eta (x) dx=1$. For $\varepsilon>0$, the rescaled mollifier $\eta_{\varepsilon}(x)=\frac{1}{\varepsilon^d}\eta(\frac{x}{\varepsilon})$, and for  any function $f\in L^1_{loc}(\mathbb{R}^d)$, its mollified version is defined by
$$f^\varepsilon(x)=(f*\eta_{\varepsilon})(x)=\int_{\mathbb{R}^d}f(x-y)\eta_{\varepsilon}(y)dy,\ \ x\in \mathbb{R}^d.$$

Subsequently, we collect some Lemmas which will be used in the present paper.
\begin{lemma}[Bernstein inequality]\label{berinequ}
	For any $b\geq a \geq 1$, there holds that
	\begin{equation}
		\norm{\dot{\Delta}_{j}f}_{L^b}\leq 2^{jd(\frac{1}{a}-\frac{1}{b})}\norm{\dot{\Delta}_{j}f}_{L^a}.
	\end{equation}
\end{lemma}
The following lemma gives a characterization of Besov spaces in terms of finite differences in the same spirit of \cite{[Shvydkoy2009],[Yu],[DR]}.
 \begin{lemma}\label{lem2.1}
Let $\alpha\in (0,1)$ and $p\in [1, \infty]$, then for any $f\in \mathcal{S}^{'}_h$, there holds
\begin{equation}\label{b1}
 f\in {\dot{B}^\alpha_{p,\infty}}\Longleftrightarrow ess\sup_{|y|>0}\frac{ \|f(\cdot-y)-f(\cdot)\|_{L^{p}}}{|y|^\alpha}<\infty ;
 \end{equation}
\begin{equation}\label{b2}
\lim_{j\rightarrow \infty} 2^{j\alpha}\norm{\dot{\Delta}_j f}_{L^p}=0 \Longleftrightarrow \lim\limits_{|y|\rightarrow 0}\frac{ \|f(\cdot-y)-f(\cdot)\|_{L^{p}}}{|y|^\alpha}=0.
\end{equation}
\end{lemma}

\begin{proof}
(1) \textquotedblleft LHS$\Rightarrow$ RHS" in \eqref{b1}. First, by the mean value theorem, the Bernstein inequality and the  Minkowski inequality, we have for any tempered distribution $f$ in $\mathbb{R}^d$,
\begin{equation}\label{2.2}
\|f(\cdot-y)-f(\cdot)\|_{L^{p}}\leq C\left(\sum_{j\leq N}2^{j}|y|\|\dot{\Delta}_{j}f\|_{L^{p}}+\sum_{j>N}\|\dot{\Delta}_{j}f\|_{L^{p}}\right).
\end{equation}
Before going further, just as \cite{[CCFS]}, we set the following localized kernel
$$K(j)=\left\{\begin{aligned}
	&2^{j\alpha},~~~~~~~~\text{if}~j\leq0,\\
	& 2^{-(1-\alpha)j},~\text{if}~j>0,
\end{aligned}\right.
$$
and denote $d_{j}=2^{j\alpha}\|\dot{\Delta}_{j}f\|_{L^{p}}$.

As a consequence, using the mean value theorem, the Bernstein inequality and the  Minkowski inequality again, we rewrite \eqref{2.2} as
$$\ba
&\|f(\cdot-y)-f(\cdot)\|_{L^{p}}\\
\leq&C\left( 2^{N(1-\alpha)}|y|\sum_{j\leq N}2^{-(N-j)(1-\alpha)}2^{j\alpha}\|\dot{\Delta}_{j}f\|_{L^{p}}+2^{-\alpha N}\sum_{j>N} 2^{(N-j)\alpha} 2^{j\alpha}\|\dot{\Delta}_{j}f\|_{L^{p}}\right)\\
\leq&C \left(2^{N(1-\alpha)}|y| +2^{-\alpha N}  \right)(K\ast d_j)(N).\ea$$
Thanks to the localized kernel $K(\cdot)\in l^{1}$, by choosing $N$ such that  $$2^{N (1-\alpha)}|y|\approx  2^{-N\alpha},$$
we arrive at
\begin{equation}\label{b4}
\|f(\cdot-y)-f(\cdot)\|_{L^{p}}\leq C|y|^\alpha (K\ast d_j)(N)\leq C|y|^\alpha \sup_{j\in \mathbb{Z}} d_j.
\end{equation}
This implies that $ess\sup_{|y|>0}\frac{ \|f(\cdot-y)-f(\cdot)\|_{L^{p}}}{|y|^\alpha}\leq C\|f\|_{\dot{B}^\alpha_{p,\infty}}$.

Now we will prove the reverse inequality \textquotedblleft LHS $\Longleftarrow$ RHS" in \eqref{b1}. As the mean value of the function $h$ is $0$, we can write
\begin{equation}
	\begin{aligned}
		\dot{\Delta}_{j} f(x)&=2^{jd}\int h(2^j y)f(x-y)dy\\
		&=2^{jd}\int h(2^j y) \left(f(x-y)-f(x)\right) dy,
	\end{aligned}
\end{equation}
which together with Minkowski inequality yields that
\begin{equation}\label{b5}
	\begin{aligned}
		2^{j\alpha}\|\dot{\Delta}_j f\|_{L^p}&\leq 2^{jd}\int 2^{j\alpha}|h(2^j y)|\|f(\cdot-y)-f(\cdot)\|_{L^p} dy\\
		&\leq ess\sup_{y\in \mathbb{R}^d} \frac{\|f(\cdot-y)-f(\cdot)\|_{L^p} }{|y|^\alpha}\int 2^{j(\alpha+d)}|y|^\alpha|h(2^j y)|dy\\
		&\leq C\,ess\sup_{y\in \mathbb{R}^d} \frac{\|f(\cdot-y)-f(\cdot)\|_{L^p} }{|y|^\alpha}.
	\end{aligned}
\end{equation}
Then we conclude the result \eqref{b1}.

(2) The proof of \eqref{b2} is similar to \eqref{b1}. Due to \eqref{b4} and \eqref{b5}, we immediately obtain the result \eqref{b2} after taking the limits.
We complete the proof of this lemma.
\end{proof}
As a corollary of Lemma \ref{lem2.1}, we have the following properties.
\begin{coro}\label{coro}
Let  $0<\alpha, \beta<1$ and  $1\leq q_1,q_2\leq \infty$. Assume that $f\in  \dot{B}^{\alpha}_{q_{1},\infty} $ and $g\in \dot{B}^{\beta}_{q_{2},c(\mathbb{N})} $, then we have
\begin{equation}\label{b5-1}
	\begin{aligned}
		&f\in \dot{B}^\alpha_{q_1,\infty}\Leftrightarrow ess\sup_{|y|>0}\frac{ \|f(\cdot-y)-f(\cdot)\|_{L^{q_1}}}{|y|^\alpha}<\infty,\\
		& g\in \dot{B}^{\beta}_{q_{2},c(\mathbb{N})} \Leftrightarrow  ess\sup_{|y|>0}\frac{ \|g(\cdot-y)-g(\cdot)\|_{L^{q_2}}}{|y|^\beta}<\infty~\text{and}~\lim\limits_{|y|\rightarrow 0}\frac{ \|g(\cdot-y)-g(\cdot)\|_{L^{q_2}}}{|y|^\beta}=0.
	\end{aligned}
\end{equation}
Moreover, there hold
	\begin{equation}\label{b5-2}\begin{aligned}
	&\|f(\cdot-y)-f(\cdot)\|_{L^{q_{1}}} =\text{O}(|y|^{\alpha})\norm{f}_{\dot{B}^\alpha_{q_1,\infty}},\\
	&\|g(\cdot-y)-g(\cdot)\|_{L^{q_{2}}} =\text{o}(|y|^{\beta})\norm{g}_{\dot{B}^\beta_{q_2,c(\mathbb{N})}}.
\end{aligned}\end{equation}
\end{coro}
Combining the properties of mollifier with Corollary \ref{coro}, we derive the following Lemma.
\begin{lemma}\label{lem2.2}
Let $\alpha, \beta\in (0,1)$,  $ p,q\in [1,\infty]$,  and $k\in \mathbb{N}^+$. Assume that  $f\in L^p(0,T;\dot{B}^\alpha_{q,\infty})$, $g\in L^p(0,T;\dot{B}^\beta_{q,c(\mathbb{N})})$,  then there hold
 \begin{enumerate}[(1)]
 \item $ \|f^{\varepsilon} -f \|_{L^{p}(0,T;L^{q})}\leq \text{O}(\varepsilon^{\alpha})\|f\|_{L^p(0,T;\dot{B}^\alpha_{q,\infty})}$;
   \item   $ \|\nabla^{k}f^{\varepsilon}  \|_{L^{p}(0,T;L^{q})}\leq \text{O}(\varepsilon^{\alpha-k})\|f\|_{L^p(0,T;\dot{B}^\alpha_{q,\infty})}$;
       \item $ \|g^{\varepsilon} -g \|_{L^{p}(0,T;L^{q})}\leq \text{o}(\varepsilon^{\beta})\|g\|_{L^p(0,T;\dot{B}^\beta_{q,c(\mathbb{N})})}$;
   \item   $ \|\nabla^{k}g^{\varepsilon}  \|_{L^{p}(0,T;L^{q})}\leq \text{o}(\varepsilon^{\beta-k})\|g\|_{L^p(0,T;\dot{B}^\beta_{q,c(\mathbb{N})})}.$
 \end{enumerate}
\end{lemma}

\begin{proof}
(1) Since $\int_{\mathbb{R}^{d}}\eta_{\varepsilon}(y)dy=1$, we deduce from direct calculations that
$$\ba
f^{\varepsilon}(x)-f(x)
=&\int f(x-y)\eta_{\varepsilon}(y)dy
-f(x)\int\eta_{\varepsilon}(y)dy\\
=&\int[f(x-y)-f(x)]\eta_{\varepsilon}(y)dy.
\ea$$
According to the Minkowski inequality and Corollary \ref{coro}, we see that
$$ \ba
\|f^{\varepsilon} -f \|_{L^{p}(0,T;L^{q})}
\leq& C\int_{B(0,\varepsilon)} \| f(x-y)-f(x)\|_{L^{p}(0,T;L^{q})} \eta_{\varepsilon}(y)dy\\
\leq& \int_{B(0,\varepsilon)}  O(|y|^{\alpha})\|f\|_{L^p(0,T;\dot{B}^\alpha_{q,\infty})} \eta_{\varepsilon}(y)dy\\
\leq& \int_{B(0,\varepsilon)}  O(\varepsilon^{\alpha})\|f\|_{L^p(0,T;\dot{B}^\alpha_{q,\infty})} \eta_{\varepsilon}(y)dy\\
\leq& O(\varepsilon^{\alpha})\|f\|_{L^p(0,T;\dot{B}^\alpha_{q,\infty})},
\ea$$
which concludes (1).

(2) Some straightforward computations yield that
$$\ba
\nabla^k f^\varepsilon(x)=\nabla^k(f\ast\eta_{\varepsilon})(x)=&\int_{B(0,\varepsilon)}\nabla^k _{x}f(x-y)\eta_{\varepsilon}(y)dy \\
=&(-1)^k\int_{B(0,\varepsilon)}\nabla^k _{y}f(x-y)\eta_{\varepsilon}(y)dy\\
=&{\varepsilon}^{-k}\int_{B(0,\varepsilon)} f(x-y)\nabla^k \eta_{\varepsilon}(y)dy.
\ea$$
Using the fact $\int_{B(0,\varepsilon)}\nabla^k\eta (y) dy=0$, we infer that
\be
{\varepsilon}^{-k}\int \nabla^k  \eta_{\varepsilon}(y)dy=0.
\ee
Hence, we arrive at
$$
\nabla f^\varepsilon(x)={\varepsilon}^{-k}\int_{B(0,\varepsilon)}[f(x-y)-f(x)]\nabla^k \eta_{\varepsilon}(y)dy.
$$
Then making use of the  Minkowski inequality and Corollary \ref{coro} once again,
we know that
 $$\ba
 \|\nabla^{k}f^{\varepsilon}  \|_{L^{p}(0,T;L^{q})}\leq & C{\varepsilon}^{-k}\int_{B(0,\varepsilon)}\|f(x-y)-f(x)\|_{L^{p}(0,T;L^{q})}|\nabla^k \eta_{\varepsilon}(y)|dy\\
 \leq & {\varepsilon}^{-k}\int_{B(0,\varepsilon)}O(|y|^{\alpha})\|f\|_{L^p(0,T;\dot{B}^\alpha_{q,\infty})} |\nabla^k  \eta_{\varepsilon}(y)|dy\\
\leq &  {\varepsilon}^{-k}\int_{B(0,\varepsilon)} O(\varepsilon^{\alpha})\|f\|_{L^p(0,T;\dot{B}^\alpha_{q,\infty})} |\nabla^k \eta_{\varepsilon}(y)|dy\\
\leq & O(\varepsilon^{\alpha-k})\|f\|_{L^p(0,T;\dot{B}^\alpha_{q,\infty})}.
 \ea$$
 This verifies the second part of this lemma. Exactly as the above derivation, we can finish the proof of the rest part of this lemma.
\end{proof}
\begin{lemma}\label{lem2.5}
	Let $\alpha\in (0,1)$ and $p\in [1,\infty]$, then there hold
		\begin{equation}\label{b5-3}
		\|fg\|_{{\dot{B}^\alpha_{p,\infty}}}\leq C\left(\|f\|_{L^\infty}\|g\|_{{\dot{B}^\alpha_{p,\infty}}}+\|f\|_{{\dot{B}^\alpha_{\infty,\infty}}}\|g\|_{L^p}\right),~ \text{for any}~ f\in B^\alpha_{\infty,\infty},g\in B^\alpha_{p,\infty},
	\end{equation}
and
	\begin{equation}
		\|fg\|_{{\dot{B}^\alpha_{p,c(\mathbb{N})}}}\leq C\left(\|f\|_{L^\infty}\|g\|_{{\dot{B}^\alpha_{p,c(\mathbb{N})}}}+\|f\|_{{\dot{B}^\alpha_{\infty,c(\mathbb{N})}}}\|g\|_{L^p}\right), \text{for any} ~ f\in B^\alpha_{\infty,c(\mathbb{N})},g\in B^\alpha_{p,c(\mathbb{N})}.
	\end{equation}

\end{lemma}
\begin{proof} Using the triangle inequality ,we have
	\begin{equation}
		\begin{aligned}
	&	|f(x-y)g(x-y)-f(x)g(x)|\\
		\leq & {|f(x-y)\Big(g(x-y)-g(x)\Big)|+|\Big(f(x-y)-f(x)\Big)g(x)|},
		\end{aligned}
	\end{equation}
	which together with the Minkowski inequality implies that
	\begin{equation}\label{b5-4}
		\begin{aligned}
	&\frac{\|f(\cdot-y)g(\cdot-y)-fg(\cdot)\|_{L^p}}	{|y|^\alpha}\\
\leq	& C\left(\frac{\|f(\cdot-y)\Big(g(\cdot-y)-g(\cdot)\Big)\|_{L^p}}{|y|^\alpha}+\frac{\|\Big(f(\cdot-y)-f(\cdot)\Big)g(\cdot)\|_{L^p}}{|y|^\alpha}\right)\\
\leq &C\left(\|f\|_{L^\infty}\frac{\|g(\cdot-y)-g(\cdot)\|_{L^p}}{|y|^\alpha}+\frac{\|f(\cdot-y)-f(\cdot)\|_{L^\infty}}{|y|^\alpha}\|g\|_{L^p}\right)\\
\leq &C \left(\|f\|_{L^\infty}\|g\|_{{\dot{B}^\alpha_{p,\infty}}}+\|f\|_{{\dot{B}^\alpha_{\infty,\infty}}}\|g\|_{L^p}\right),
		\end{aligned}
	\end{equation}
where we have used the  Corollary \ref{coro}.
Moreover, taking the limits on \eqref{b5-4} and using the Corollary \ref{coro} again, we have
\begin{equation}
	\begin{aligned}
		&\lim\limits_{|y|\rightarrow0}\frac{\|f(\cdot-y)g(\cdot-y)-fg(\cdot)\|_{L^p}}	{|y|^\alpha}\\
		&\leq C\left(\|f\|_{L^\infty}	\lim\limits_{|y|\rightarrow0}\frac{\|g(\cdot-y)-g(\cdot)\|_{L^p}}{|y|^\alpha}+	\lim\limits_{|y|\rightarrow0}\frac{\|f(\cdot-y)-f(\cdot)\|_{L^\infty}}{|y|^\alpha}\|g\|_{L^p}\right)\\
		&\leq C\left(\|f\|_{L^\infty}\|g\|_{{\dot{B}^\alpha_{p,c(\mathbb{N})}}}+\|f\|_{{\dot{B}^\alpha_{\infty,c(\mathbb{N})}}}\|g\|_{L^p}\right).
	\end{aligned}
\end{equation}
\end{proof}
Next, we will state a new Constantin-E-Titi type commutator estimate as follows.
\begin{lemma}	\label{lem2.3}
	Assume that $0<\alpha,\beta<1$, $1\leq p,q,p_{1},p_{2}\leq\infty$ and $\frac{1}{p}=\frac{1}{p_1}+\frac{1}{p_2}$.
Then, there holds
	\begin{align} \label{cet}
		\|(fg)^{\varepsilon}- f^{\varepsilon}g^{\varepsilon}\|_{L^p(0,T;L^q)} \leq  \text{o}(\varepsilon^{\alpha+\beta}),	
	\end{align}
provided that one of the following three conditions is satisfied,
\begin{enumerate}[(1)]
 \item  $f\in L^{p_1}(0,T;\dot{B}^{\alpha}_{q_{1},c(\mathbb{N})} )$, $g\in L^{p_2}(0,T;\dot{B}^{\beta}_{q_{2},\infty} )$, $1\leq q_{1},q_{2}\leq\infty$, $\frac{1}{q}=\frac{1}{q_1}+\frac{1}{q_2}$;
  \item  $\nabla f\in   L^{p_1}(0,T;\dot{B}^{\alpha}_{q_{1},c(\mathbb{N})} )$, $\nabla g\in L^{p_2}(0,T;\dot{B}^{\beta}_{q_{2},\infty} )$,  $\f{2}{d}+\f1q=\frac{1}{q_{1}}+\frac{1}{q_{2}}$, $1\leq q_{1},q_{2}<d$;
  \item  $  f\in   L^{p_1}(0,T;\dot{B}^{\alpha}_{q_{1},c(\mathbb{N})} )$, $\nabla g\in L^{p_2}(0,T;\dot{B}^{\beta}_{q_{2},\infty} )$,  $\f{1}{d}+\f1q=\frac{1}{q_{1}}+\frac{1}{q_{2}}$, $1\leq q_{2}<d$,  $1\leq q_{1}\leq\infty$.
 \end{enumerate}\end{lemma}
 \begin{remark}
  The estimate that $\int_0^T\int | (fg)^{\varepsilon}- f^{\varepsilon}g^{\varepsilon}||\nabla h |dxdt$ frequently appears in the study of energy (helicity) conservation  of fluid  equations
    (see \cite{[CET],[WYY],[YWW],[CY],[De Rosa],[WY]}).
 Therefore, it seems that this lemma is very helpful in this research of weak solutions in critical Besov spaces.
 \end{remark}
 \begin{remark}\label{rem2.3} We would like to point out that $o(\cdot)$ should be replaced by $O(\cdot)$ if the space $\dot{B}^{\alpha}_{q_{1},c(\mathbb{N})}$  in the condition of this lemma is replaced by the space $\dot{B}^{\alpha}_{q_{1},\infty}$. \end{remark}
 \begin{remark}\label{rem2.4} It is worth remarking that $\nabla$ may be replaced by curl for incompressible flow if $1<q_{1},q_{2}<d$.
 \end{remark}
\begin{remark}
The results of  Lemma \ref{lem2.1}-\ref{lem2.3}   still  hold for the nonhomogeneous Besov spaces.
\end{remark}
\begin{proof}
First, we recall the following   identity observed  by Constantin-E-Titi   in \cite{[CET]} that
	$$\ba&(fg)^{\varepsilon}(x)- f^{\varepsilon}g^{\varepsilon}(x)\\
=&
\int_{\Omega}\eta_{\varepsilon}(y)
[f(x-y)-f(x)][g(x-y)-g(x)]dy-
(f-f^{\varepsilon})(g-g^{\varepsilon})(x).
	\ea$$
(1)	Using the H\"older's inequality and Minkowski inequality, we obtain
	$$\begin{aligned} & \|(fg)^{\varepsilon}- f^{\varepsilon}g^{\varepsilon}\|_{L^p(0,T;L^q)}\\ \leq&C\int_{|y|\leq\varepsilon}\eta_{\varepsilon}(y)\| f(\cdot-y)-f(\cdot)\|_{L^{p_1}(0,T;L^{q_1})} \| g(\cdot-y)-g(\cdot)\|_{L^{p_2}(0,T;L^{q_2})}dy \\&+C\| f-f^{\varepsilon}\|_{L^{p_1}(0,T;L^{q_1})}  \| g-g^{\varepsilon}\|_{L^{p_2}(0,T;L^{q_2})}\\
	\leq& o(\varepsilon^{\alpha+\beta})\| f\|_{L^{p_1}(0,T;\dot{B}^{\alpha}_{q_{1},c(\mathbb{N})})}
	\|g \|_{L^{p_2}(0,T;\dot{B}^{\beta}_{q_{2},\infty})},
	\end{aligned} $$
where Corollary \ref{coro} and Lemma \ref{lem2.2} are used.

(2)	Let $\frac{1}{q}=\frac{1}{q_1^\natural}+\frac{1}{q_2^\natural}$ and  $\frac{1}{q_i^{\natural}}=\frac{1}{q_i}-\frac{1}{d}$, $i=1,2.$ A slight variant of the above proof implies that
		$$\begin{aligned} & \|(fg)^{\varepsilon}- f^{\varepsilon}g^{\varepsilon}\|_{L^p(0,T;L^q)}\\ \leq&C\int_{|y|\leq\varepsilon}\eta_{\varepsilon}(y)\| f(\cdot-y)-f(\cdot)\|_{L^{p_1}(0,T;L^{q_1^{\natural}})} \| g(\cdot-y)-g(\cdot)\|_{L^{p_2}(0,T;L^{q_2^{\natural}})}dy \\&+C\| f-f^{\varepsilon}\|_{L^{p_1}(0,T;L^{q_1^{\natural}})}  \| g-g^{\varepsilon}\|_{L^{p_2}(0,T;L^{q_2^{\natural}})}\\
\leq&C\int_{|y|\leq\varepsilon}\eta_{\varepsilon}(y)\| \nabla f(\cdot-y)-\nabla f(\cdot)\|_{L^{p_1}(0,T;L^{q_1})} \| \nabla g(\cdot-y)-\nabla g(\cdot)\|_{L^{p_2}(0,T;L^{q_2})}dy \\&+C\| \nabla f- \nabla f^{\varepsilon}\|_{L^{p_1}(0,T;L^{q_1})}  \| \nabla g-\nabla g^{\varepsilon}\|_{L^{p_2}(0,T;L^{q_2})}\\
	\leq& o(\varepsilon^{\alpha+\beta})\| \nabla f\|_{L^{p_1}(0,T;\dot{B}^{\alpha}_{q_{1},c(\mathbb{N})})}
	\|\nabla g \|_{L^{p_2}(0,T;\dot{B}^{\beta}_{q_{2},\infty})},
	\end{aligned} $$
where we have used the H\"older inequality and Sobolev embedding theorem.

(3)	Arguing in the same manner as above, we know that
$$\begin{aligned} & \|(fg)^{\varepsilon}- f^{\varepsilon}g^{\varepsilon}\|_{L^p(0,T;L^q)}\\ \leq&C\int_{|y|\leq\varepsilon}\eta_{\varepsilon}(y)\| f(\cdot-y)-f(\cdot)\|_{L^{p_1}(0,T;L^{q_1})} \| g(\cdot-y)-g(\cdot)\|_{L^{p_2}(0,T;L^{q_2^{\natural}})}dy \\&+C\| f-f^{\varepsilon}\|_{L^{p_1}(0,T;L^{q_1})}  \| g-g^{\varepsilon}\|_{L^{p_2}(0,T;L^{q_2^{\natural}})}\\
\leq&C\int_{|y|\leq\varepsilon}\eta_{\varepsilon}(y)\| f(\cdot-y)- f(\cdot)\|_{L^{p_1}(0,T;L^{q_1})} \| \nabla g(\cdot-y)-\nabla g(\cdot)\|_{L^{p_2}(0,T;L^{q_2})}dy \\&+C\|  f- f^{\varepsilon}\|_{L^{p_1}(0,T;L^{q_1})}  \| \nabla g-\nabla g^{\varepsilon}\|_{L^{p_2}(0,T;L^{q_2})}\\
	\leq& o(\varepsilon^{\alpha+\beta})\| \nabla f\|_{L^{p_1}(0,T;\dot{B}^{\alpha}_{q_{1},c(\mathbb{N})})}
	\|\nabla g \|_{L^{p_2}(0,T;\dot{B}^{\beta}_{q_{2},\infty})},
	\end{aligned} $$
where the indices are required to satisfy $\frac{1}{q}=\frac{1}{q_1}+\frac{1}{q_2^\natural}$ and $\frac{1}{q_2^\natural}=\frac{1}{q_2}-\frac{1}{d}$ with $1\leq q_2<d$.

	Then the proof of this lemma is completed.
\end{proof}

\begin{lemma}\label{lem2.7}
	Let $ p,q,p_1,q_1,p_2,q_2\in[1, \infty)$ with $\frac{1}{p}=\frac{1}{p_1}+\frac{1}{p_2},\frac{1}{q}=\frac{1}{q_1}+\frac{1}{q_2} $. Assume that $f\in L^{p_1}(0,T;L^{q_1}(\Omega)) $ and $g\in L^{p_2}(0,T;L^{q_2}(\Omega))$. Then, as $\varepsilon\rightarrow0$, there hold
	\begin{equation}\label{b6}
		\|(fg)^\varepsilon-f^\varepsilon g^\varepsilon\|_{L^p(0,T;L^q(\Omega))}\rightarrow 0,
	\end{equation}
and
\begin{equation}\label{b7}
	\|(f\times g)^\varepsilon-f^\varepsilon\times g^\varepsilon\|_{L^p(0,T;L^q(\Omega))}\rightarrow 0.
\end{equation}
\end{lemma}
\begin{proof}
First, it follows from the triangle inequality that
	\begin{equation*}
		\begin{aligned}
			&\|(fg)^\varepsilon-f^\varepsilon g^\varepsilon\|_{L^p(0,T;L^q)}\\
			\leq & C\left(\|(fg)^\varepsilon- (fg)\|_{L^p(0,T;L^q)}+\|fg-f^\varepsilon g\|_{L^p(0,T;L^q)}+\|f^\varepsilon g-f^\varepsilon g^\varepsilon\|_{L^p(0,T;L^q)}\right)\\
			\leq &C\Big(\|(fg)^\varepsilon- fg\|_{L^p(0,T;L^q)}+\|f-f^\varepsilon\|_{L^{p_1}(0,T;L^{q_1})}\|g\|_{L^{p_2}(0,T;L^{q_2})}\\
			&\ \ \ \ \ +\|f^\varepsilon\|_{L^{p_1}(0,T;L^{q_1})}\|g-g^\varepsilon\|_{L^{p_2}(0,T;L^{q_2})}\Big),
		\end{aligned}
	\end{equation*}
	which together with the properties of the standard mollifiers yields \eqref{b6}.
	
	Furthermore, to conclude \eqref{b7}, we only consider the case that the spatial dimension $d=3$, since the proof for the case $d=2$ is similar. Without loss of generality, we assume that $f=(f_1,f_2,f_3)$ and $g=(g_1,g_2,g_3)$.
	Then by a direct computation, we have
	\begin{equation}
		\begin{aligned}\label{b8}
			&\|(f\times g)^\varepsilon-f^\varepsilon \times g^\varepsilon\|_{L^p(0,T;L^q)}=\Big\|\left(\left|
				\begin{matrix}
					\overrightarrow{i}&\overrightarrow{j}&\overrightarrow{k}\\
					f_1&f_2&f_3\\
					g_1&g_2&g_3\\
				\end{matrix}
			\right|\right)^\varepsilon-\left|
			\begin{matrix}
				\overrightarrow{i}&\overrightarrow{j}&\overrightarrow{k}\\
				f_1^\varepsilon&f_2^\varepsilon&f_3^\varepsilon\\
				g_1^\varepsilon&g_2^\varepsilon&g_3^\varepsilon\\
			\end{matrix}
			\right|\Big\|_{L^p(0,T;L^q)}\\
			=&\Big\|\left((f_2g_3-f_3g_2)\overrightarrow{i}-(f_1g_3-f_3g_1)\overrightarrow{j}+(f_1g_2-f_2g_1)\overrightarrow{k}\right)^\varepsilon\\
			&~~~~-\left((f_2^\varepsilon g_3^\varepsilon-f_3^\varepsilon g_2^\varepsilon)\overrightarrow{i}-(f_1^\varepsilon g_3^\varepsilon-f_3^\varepsilon g_1^\varepsilon)\overrightarrow{j}+(f_1^\varepsilon g_2^\varepsilon-f_2^\varepsilon g_1^\varepsilon)\overrightarrow{k}\right)\Big\|_{L^p(0,T;L^q)}\\
			=&\Big\|\left[\Big((f_2g_3)^\varepsilon-f_2^\varepsilon g_3^\varepsilon\Big)-\Big((f_3g_2)^\varepsilon-f_3^\varepsilon g_2^\varepsilon\Big)\right]\overrightarrow{i}-\left[\Big((f_1g_3)^\varepsilon-f_1^\varepsilon g_3^\varepsilon\Big)-\Big((f_3g_1)^\varepsilon-f_3^\varepsilon g_1^\varepsilon\Big)\right]\overrightarrow{j}\\
			&+\left[\Big((f_1g_2)^\varepsilon-f_1^\varepsilon g_2^\varepsilon\Big)-\Big((f_2g_1)^\varepsilon-f_2^\varepsilon g_1^\varepsilon\Big)\right]\overrightarrow{k}\Big\|_{L^p(0,T;L^q)},
		\end{aligned}
	\end{equation}
which together with the triangle inequality and \eqref{b6} leads to \eqref{b7}.
\end{proof}
For the convenience of readers, we end this section by presenting the definition of weak solutions to the compressible Euler equations \eqref{CEuler},   incompressible Euler equations  \eqref{Euler}  and the surface quasi-geostrophic equation
 \eqref{qg}, respectively.
 \begin{definition}\label{ceulerdefi}
 	A pair ($\rho,v$) is called a weak solution to  the   compressible Euler equations \eqref{CEuler},  if ($\rho,v$) satisfies

	\begin{enumerate}[(i)]
		\item for any test function $\phi\in C_0^\infty((0,T)\times \Omega)$, there holds
		$$\int^T_0\int_{\Omega}\rho(x,t)\partial_t \phi(x,t)+ \rho(x,t)v(x,t)\nabla \phi(x,t)dxdt=0.$$
		\item
  for any  test vector field $\varphi\in C_{0}^{\infty}((0,T)\times\Omega)$, there holds
		\begin{equation*}
			\begin{aligned}
 \int_{0}^{T}\int_{\Omega}\rho(x,t)v(x,t)\partial_{t}\varphi(x,t)&+(\rho(x,t)v(x,t)\otimes v(x,t)) \nabla\varphi(x,t)\\+&
			\pi(\rho )(x,t)\text{div} \varphi(x,t)dxdt=0.
				\end{aligned}	\end{equation*}
	\item
	the energy inequality holds
	\begin{equation}\label{energyineq}
		\begin{aligned}
			\mathcal{E}(t)   \leq \mathcal{E}(0), 		\end{aligned}\end{equation}
	where $\mathcal{E}(t)=\int_{\Omega}\left( \frac{1}{2}\rho |v|^2+\kappa\f{\rho^{\gamma}}{\gamma-1} \right) dx$.
		
	\end{enumerate}
\end{definition}
 \begin{definition}\label{eulerdefi}
	A vector field $v\in C_{\text{weak}}([0,T];L^{2}(\Omega))$ is called  a weak solution of the Euler equations \eqref{Euler} with initial data $v_{0}\in L^{2}(\Omega)$, if  $ v$  satisfies
	\begin{enumerate}[(i)]
		\item  for any divergence-free test function $\varphi\in C_{0}^{\infty}((0,T)\times\Omega)$, there holds
		$$
		\int_{0}^{T}\int_{\Omega}v(x,t)\partial_{t}\varphi(x,t)+v(x,t)\otimes v(x,t)\cdot\nabla\varphi(x,t)dxdt=0.
		$$
		\item[(ii)]
		$v$ is weakly divergence-free, namely, for every scalar test function $\psi\in C_{0}^{\infty}((0,T)\times\Omega)$,
		$$\int_0^T\int_{\Omega} \Div v\cdot\psi dxdt=0.$$
		
	\end{enumerate}
\end{definition}
\begin{definition}\label{qgdefi}
	A vector field $\theta\in C([0,T];L^{2}(\Omega))$  is called  a weak solution of the 2-D  quasi-geostrophic equation \eqref{qg} with initial data $\theta_{0}\in L^{2}(\Omega)$, if there hold	
\begin{equation}
	\int_{\Omega}[\theta(x,T) \varphi                                     (x,T)- \theta(x,0) \varphi(x,0)]dx =\int_{0}^{T}\int_{\Omega}\theta(x,t)(\partial_{t}\varphi(x,t)+v\cdot\nabla\varphi)dxdt,
\end{equation}
	and
	\begin{equation}\label{1.9}
		v(x,t)=-\nabla^{\perp}\int_{\Omega} \frac{\theta(y,t)}{|x-y|}dy
	\end{equation}
	for any   test function $\varphi\in C_{0}^{\infty}([0,T];C^{\infty}(\Omega))$, where $\nabla^{\perp}$ in \eqref{1.9} is in the sense of
	distributions.
\end{definition}
\section{ Helicity  conservation for the compressible Euler equations}
This section contains the proof of helicity conservation of weak solutions in Onsager type spaces $\dot{B}^{\f13}_{p,c(\mathbb{N})}$ for the  compressible Euler equations.
\begin{prop}\label{propo1}
	Let $(\rho,v)$ be a weak solution of compressible Euler equations \eqref{CEuler} in the sense of Definition \ref{ceulerdefi}, and $\text{div\,}v,\text{curl\,}v\in C([0,T];L^{\f{2d}{d+1}}(\Omega))$. If $(\rho, v)$ additionally satisfies
 $$\ba
 & 0<c_1\leq \rho \leq c_2<\infty, ~\rho \in L^3(0,T;\dot{B}^{\frac{1}{3}}_{3,c(\mathbb{N})}),~\rho v \in L^3(0,T;\dot{B}^{\frac{1}{3}}_{3,\infty}),\\
  &v \in L^3(0,T;B^{\frac{1}{3}}_{3,c(\mathbb{N})}),
 \ea$$
  then $(\rho,v)$ is also a weak solution of system \eqref{rCEuler1} in the sense of distributions.
\end{prop}
\begin{proof}
	First,  we mollify the momentum equation $\eqref{CEuler}_2$ in space (with the kernel and notation as in Section 2):
	\begin{eqnarray}\label{p2}
		(\rho v)^\varepsilon_t+\text{div}(\rho v\otimes v)^\varepsilon +\nabla \pi(\rho)^\varepsilon=0.
	\end{eqnarray}
	Then, in virtue of appropriate commutators, the above equation can be rewritten as
	\begin{equation}
		\begin{aligned}
			(\rho^\varepsilon v^\varepsilon)_t &+\text{div}((\rho  v)^\varepsilon\otimes v^\varepsilon)+\nabla \pi(\rho^\varepsilon)\\
			&=-[(\rho v)^\varepsilon-(\rho^\varepsilon v^\varepsilon)]_t-\text{div}[(\rho v\otimes v)^\varepsilon-(\rho  v)^\varepsilon \otimes v^\varepsilon]-\nabla [\pi(\rho^\varepsilon)-\pi(\rho)^\varepsilon].
		\end{aligned}
	\end{equation}
	Taking into account of  mollified version of the continuity
	equation
	\begin{equation}\label{p1}
		\rho^\varepsilon_t+\text{div}(\rho v)^\varepsilon=0,
	\end{equation}
	a direct computation shows
	\begin{equation}\label{p3}
		\begin{aligned}
			\rho^\varepsilon v^\varepsilon_t+\rho^\varepsilon v^\varepsilon \cdot \nabla v^\varepsilon+\nabla \pi(\rho^\varepsilon)=&-[ (\rho v)^\varepsilon-\rho^\varepsilon v^\varepsilon] \cdot \nabla v^\varepsilon
			-[(\rho v)^\varepsilon-(\rho^\varepsilon v^\varepsilon)]_t\\
			&-\text{div}[(\rho v\otimes v)^\varepsilon-(\rho  v)^\varepsilon \otimes v^\varepsilon]-\nabla [\pi(\rho^\varepsilon)-\pi(\rho)^\varepsilon].
		\end{aligned}
	\end{equation}
	Since $0<c_1\leq \rho \leq c_2<\infty$, setting $\Pi(\rho^\varepsilon )= \int_{1}^{\rho^\varepsilon}\f{\pi'(s)}{s}ds$ and  dividing both sides of equation \eqref{p3} by $\rho^\varepsilon$, we arrive at
	\begin{equation}\label{p4}
		\begin{aligned}
			&v^\varepsilon_t+\omega^\varepsilon \times v^\varepsilon+\nabla\left( \Pi(\rho^\varepsilon)+\frac{1}{2}|v^\varepsilon|^2\right)\\
			=&-\frac{1}{\rho^\varepsilon}[ (\rho v)^\varepsilon-\rho^\varepsilon v^\varepsilon] \cdot \nabla v^\varepsilon
			-\frac{1}{\rho^\varepsilon}[(\rho v)^\varepsilon-(\rho^\varepsilon v^\varepsilon)]_t\\
			&-\frac{1}{\rho^\varepsilon}\text{div}[(\rho v\otimes v)^\varepsilon-(\rho  v)^\varepsilon \otimes v^\varepsilon]
			-	\frac{1}{\rho^\varepsilon}\nabla [\pi(\rho^\varepsilon)-\pi(\rho)^\varepsilon],
		\end{aligned}
	\end{equation}
	where we have used the identity that $v^\varepsilon\cdot \nabla v^\varepsilon=\frac{1}{2}\nabla |v^\varepsilon|^2 +\omega^\varepsilon \times v^\varepsilon$ with $\omega =\text{curl\,} v$.
	
	Let $\varphi(x,t)\in C_0^\infty((0,T)\times \Omega)$ be a test function. Multiplication with $\varphi$ and then integration yield that
	\begin{equation}\label{p5}
		\begin{aligned}
			\int_0^t \int_{\Omega}& \left(v^\varepsilon_\tau+\omega^\varepsilon \times v^\varepsilon+\nabla\left( \Pi(\rho^\varepsilon)+\frac{1}{2}|v^\varepsilon|^2\right)\right)\varphi  dxd\tau\\
			=&\int_0^t\int_{\Omega} \Big(-\frac{1}{\rho^\varepsilon}[ (\rho v)^\varepsilon-\rho^\varepsilon v^\varepsilon] \cdot \nabla v^\varepsilon
			-\frac{1}{\rho^\varepsilon}[(\rho v)^\varepsilon-(\rho^\varepsilon v^\varepsilon)]_\tau\\
			&-\frac{1}{\rho^\varepsilon}\text{div}[(\rho v\otimes v)^\varepsilon-(\rho  v)^\varepsilon \otimes v^\varepsilon]
			-	\frac{1}{\rho^\varepsilon}\nabla [\pi(\rho^\varepsilon)-\pi(\rho)^\varepsilon]\Big)\varphi dxd\tau\\
			=&I_1+I_2+I_3+I_4.
		\end{aligned}
	\end{equation}
	Next we will show that the four terms $I_1-I_4$ on the RHS of \eqref{p5} tend to zero as $\varepsilon \rightarrow 0$.
	
	Indeed, note that
 $$
 \rho \in L^3(0,T;\dot{B}^{\frac{1}{3}}_{3,c(\mathbb{N})}), \rho v \in L^3(0,T;\dot{B}^{\frac{1}{3}}_{3,\infty}), v\in L^3(0,T;{B}^{\frac{1}{3}}_{3,c(\mathbb{N})}).
 $$
	Using Corollary \ref{coro}, Lemma \ref{lem2.2} and the H\"older inequality, the first term $I_1$ on the RHS of \eqref{p5} can be handled as
	\begin{equation}\label{p6}
		\begin{aligned}
			|I_1|&\leq C \|(\rho v)^\varepsilon-\rho^\varepsilon v^\varepsilon\|_{L^{\frac{3}{2}}(0,T;L^{\frac{3}{2}}(\Omega))}\|\nabla v^\varepsilon\|_{L^3(0,T;L^3(\Omega))}\|\varphi\|_{L^\infty(0,T;L^\infty(\Omega))}\\
			&\leq o(1)\|\rho \|_{L^3(0,T;\dot{B}^{\frac{1}{3}}_{3,c(\mathbb{N})})}\|v\|_{L^3(0,T;\dot{B}^{\frac{1}{3}}_{3,c(\mathbb{N})} )}^2\\
			&\leq o(1).
		\end{aligned}
	\end{equation}
	For second term $I_2$, using integration by parts, conbination the Corollary \ref{coro} with Lemma \ref{lem2.2}, we have
	\begin{equation}\label{p7}
		\begin{aligned}
			|I_2|&=\left|\int_0^t\int_{\Omega} [(\rho v)^\varepsilon-(\rho^\varepsilon v^\varepsilon)]\left(-\frac{\partial_\tau \rho^\varepsilon}{(\rho^\varepsilon)^2}\varphi+\frac{1}{\rho^\varepsilon}\partial_\tau \varphi\right)dxd\tau\right|\\
			&=\left|\int_0^t\int_{\Omega} [(\rho v)^\varepsilon-(\rho^\varepsilon v^\varepsilon)]\left(\frac{\text{div}( \rho v)^\varepsilon}{(\rho^\varepsilon)^2}\varphi+\frac{1}{\rho^\varepsilon}\partial_\tau \varphi\right)dxd\tau\right|\\
			&\leq C\|(\rho v)^\varepsilon-\rho^\varepsilon v^\varepsilon\|_{L^{\frac{3}{2}}(0,T;L^{\frac{3}{2}}(\Omega))}\left(\|\nabla (\rho v)^\varepsilon\|_{L^3(0,T;L^3(\Omega))}+\|\partial_\tau \varphi\|_{L^3(0,T;L^3(\Omega))}\right)\\
			&\leq o(\varepsilon^{\frac{2}{3}})\|\rho \|_{L^3(0,T;\dot{B}^{\frac{1}{3}}_{3,c(\mathbb{N})} )}\|v\|_{L^3(0,T;{\dot{B}^{\frac{1}{3}}_{3,c(\mathbb{N})}})}\big(O(\varepsilon^{\frac{1}{3}-1})\|\rho v\|_{L^3(0,T;\dot{B}^{\frac{1}{3}}_{3,\infty} )}+1\big)\\
			&\leq o(1)+o(\varepsilon^{\frac{2}{3}}).
		\end{aligned}
	\end{equation}
	The third term $I_3$ can be estimated similarly as $I_2$:
	\begin{eqnarray}\label{p8}
		\begin{aligned}
			|I_3|&=\left|\int_0^t\int_{\Omega} [(\rho v\otimes v)^\varepsilon-(\rho v)^\varepsilon\otimes v^\varepsilon]\left(-\frac{\nabla \rho ^\varepsilon}{(\rho ^\varepsilon)^2}\varphi+\frac{\nabla \varphi}{\rho ^\varepsilon}\right)	dxd\tau \right|\\
			&\leq C\|(\rho v\otimes v)^\varepsilon-(\rho v)^\varepsilon \otimes v^\varepsilon\|_{L^{\frac{3}{2}}(0,T;L^{\frac{3}{2}}(\Omega))}\left(\|\nabla \rho ^\varepsilon\|_{L^3(0,T;L^3(\Omega))}+\|\nabla \varphi\|_{L^3(0,T;L^3(\Omega))}\right)\\
			&\leq o(\varepsilon^{\frac{2}{3}})\|\rho v\|_{L^3(0,T;\dot{B}^{\frac{1}{3}}_{3,\infty})}\|v\|_{L^3(0,T;\dot{B}^{\frac{1}{3}}_{3,c(\mathbb{N})})}\left(o(\varepsilon^{\frac{1}{3}-1})\|\rho\|_{L^3({\dot{B}^{\frac{1}{3}}_{3,c(\mathbb{N})}})}+1\right)\\
			&\leq o(1)+o(\varepsilon^{\frac{2}{3}}).
		\end{aligned}
	\end{eqnarray}
	For the last term $I_4$, notice that if  $\pi(s)\in C^2([a,b])$, then for any $s,s_0\in [a,b]$ there holds
	$$|\pi(s)-\pi(s_0)-\pi^{'}(s_0)(s-s_0)|\leq C(s-s_0)^2,$$
	where the constant $C$ can be independent of $s$ and $s_0$.
	This yields that
	\begin{equation}\label{p9}
		|\pi(\rho^\varepsilon(t,x))-\pi(\rho(t,x))-\pi^{'}(\rho (t,x))(\rho^\varepsilon(t,x)-\rho(t,x))|\leq C|\rho^\varepsilon(t,x)-\rho(t,x)|^2.
	\end{equation}
	Similarly, we also have
	\begin{equation}
		|\pi(\rho(t,y))-\pi(\rho(t,x))-\pi^{'}(\rho (t,x))(\rho (t,y)-\rho (t,x))|\leq C|\rho (t,y)-\rho (t,x)|^2.
	\end{equation}
	Applying convolution with respect to $y$ to the last inequality and invoking Jensen's inequality, we get
	\begin{equation}\label{p10}
		\begin{aligned}
			|\pi^\varepsilon(\rho(t,x))-\pi(\rho(t,x))-\pi^{'}(\rho (t,x))(\rho^\varepsilon 	 (t,x)-\rho (t,x))|\leq C(\rho (t,\cdot)-\rho (t,x))^2*\eta_{\varepsilon}(x),
		\end{aligned}
	\end{equation}
	which together with \eqref{p9} gives
	\begin{equation}\label{pi}
		\begin{aligned}
			|\pi(\rho^\varepsilon(t,x))-\pi^\varepsilon(\rho(t,x))|&\leq C|\rho^\varepsilon(t,x)-\rho(t,x)|^2+C(\rho (t,\cdot)-\rho (t,x))^2*\eta_{\varepsilon}(x).
		\end{aligned}
	\end{equation}
	Then it follows from integration by parts that $I_4$ can be handled as
	\begin{equation}\label{p11}
		\begin{aligned}
			|I_4|=&\left|\int_0^t\int_{\Omega} [\pi(\rho^\varepsilon)-\pi^\varepsilon(\rho)]\left(\frac{1}{\rho^\varepsilon}\text{div} \varphi-\varphi\cdot \frac{\nabla \rho^\varepsilon}{(\rho^\varepsilon)^2}\right)dxd\tau\right|\\
\leq& C\|\rho^\varepsilon(t,x)-\rho(t,x)|^2+(\rho (t,\cdot)-\rho (t,x))^2*\eta_{\varepsilon}(x)\|_{L^{\frac{3}{2}}(0,T;L^{\frac{3}{2}}(\Omega))}\\&\times\left(\|\nabla \varphi\|_{L^3(0,T;L^3(\Omega))}+\|\nabla \rho^\varepsilon\|_{L^3(0,T;L^3(\Omega))}\right)\\
\leq&  o(\varepsilon^{\frac{2}{3}})\|\rho\|_{L^3(0,T;\dot{B}^{\frac{1}{3}}_{3,c(\mathbb{N})})}^2
\left(1+o(\varepsilon^{\frac{1}{3}-1})
\|\rho\|_{L^3(0,T; \dot{B}^{\frac{1}{3}}_{3,c(\mathbb{N}) })}\right)\\
\leq & o(\varepsilon^{\frac{2}{3}})+o(1).
		\end{aligned}
	\end{equation}
	Hence, letting $\varepsilon\rightarrow 0$ and using the Lebesgue's dominated convergence theorem for LHS of \eqref{p5}, we have
	\begin{equation}
		\int_0^t \int_{\Omega} \left(v_\tau+\omega \times v+\nabla\left( \Pi(\rho)+\frac{1}{2}|v|^2\right)\right)\varphi  dxd\tau=0,
	\end{equation}
	which means that $(\rho,v)$ is also a weak solution of equations
	\begin{equation}\label{3.16}\left\{\ba
	&\rho_t+\nabla \cdot (\rho v)=0, \\
	& v_t+\omega \times v+\nabla\left( \Pi(\rho)+\frac{1}{2}|v|^2\right)=0.
	\ea\right.\end{equation}
	This concludes the proof of this proposition.
\end{proof}
\begin{proof}[Proof of Theorem \ref{the01}]
	Let $(\rho,v)$ be a weak solution to the compressible Euler equations \eqref{CEuler}. Combining \eqref{a1} and Proposition \ref{propo1}, we know that $(\rho,v)$ is also a weak solution to the system \eqref{3.16}. Then one can  mollify the equation \eqref{3.16} in space direction to deduce that
	$$\left\{\begin{aligned}
		&v^{\varepsilon}_{t} +(\omega\times v)^{\varepsilon}+\nabla(\Pi(\rho )+\f12|v|^{2})^{\varepsilon}= 0,\\
		&\omega^{\varepsilon}_{t}-\text{curl}( v\times\omega)^{\varepsilon}=0,\\
		& \text{div}\omega^{\varepsilon} =0.
	\end{aligned}\right.$$
	Consequently, by a straightforward computation, we have
	\be\label{ce3.1}\ba
	\f{d}{dt}\int_{\Omega}(v^{\varepsilon}\cdot\omega^{\varepsilon})dx=&\int_{\Omega} v_{t}^{\varepsilon}\cdot
	\omega^{\varepsilon}+\omega_{t}^{\varepsilon}\cdot v^{\varepsilon}dx\\
	=&\int_{\Omega}-\left((\omega\times v)^{\varepsilon}+\nabla(\Pi(\rho )+\f12|v|^{2})^{\varepsilon}\right)\cdot \omega^{\varepsilon}+
	\left(\text{curl}\left( v\times\omega\right)\right)^{\varepsilon}\cdot v^{\varepsilon}dx\\
	=&\int_{\Omega}-\left(\omega\times v\right)^{\varepsilon}\cdot \omega^{\varepsilon}+
	\left(   v\times\omega \right)^{\varepsilon}\cdot\omega^{\varepsilon}dx
	\\
	=&\int_{\Omega}-2\Big((\omega\times v)^{\varepsilon} -\omega^{\varepsilon}\times v^{\varepsilon} \Big)\cdot\omega^{\varepsilon}-
	2(\omega^{\varepsilon}\times v^{\varepsilon} )\cdot\omega^{\varepsilon}dx\\
	=&\int_{\Omega}-2\Big((\omega\times v)^{\varepsilon} -\omega^{\varepsilon}\times v^{\varepsilon} \Big)\cdot\omega^{\varepsilon}dx\\
=&I,
	\ea\ee
where we have used the identity that $\omega^\varepsilon\times v^\varepsilon\cdot v^\varepsilon=\omega^\varepsilon\times\omega ^\varepsilon\cdot v^\varepsilon=0$.

Next we will show that $I$ tends to zero as $\varepsilon\rightarrow0$.

(1) The H\"older inequality ensures that
\be\ba
	|I|\leq &C \|(\omega \times v)^\varepsilon-\omega^\varepsilon \times v^\varepsilon\|_{L^{\frac{3}{2}}(0,T;L^{\frac{3}{2}}(\Omega))}\|\omega ^\varepsilon\|_{L^3(0,T;L^3(\Omega))}\\
\leq &C \|(\omega \times v)^\varepsilon-\omega^\varepsilon \times v^\varepsilon\|_{L^{\frac{3}{2}}(0,T;L^{\frac{3}{2}}(\Omega))}\|\nabla v ^\varepsilon\|_{L^3(0,T;L^3(\Omega))}.
\ea\ee
By $v\in L^{3}(0,T;{B}^{\f13}_{3,c(\mathbb{N})})$,
$\omega\in L^{3}(0,T;\dot{B}^{\frac{1}{3}}_{3,\infty}) $  and  Lemma \ref{lem2.3}, we deduce that
$$
\|(\omega \times v)^\varepsilon-\omega^\varepsilon \times v^\varepsilon\|_{L^{\frac{3}{2}}(0,T;L^{\frac{3}{2}}(\Omega))}\leq  o(\varepsilon^{\f23}).
$$
Combining $v\in L^{3}(0,T;{B}^{\f13}_{3,c(\mathbb{N})})$
and  (2) in Lemma \ref{lem2.2},  we observe that
$$\|\nabla v ^\varepsilon\|_{L^3(0,T;L^3(\Omega))}\leq  o(\varepsilon^{-\f23}).$$
Hence there holds
$$|I| \leq o(1).$$

(2) By the H\"older inequality and Lemma  \ref{lem2.7}, we get
\begin{equation}
|	I|\leq C \|(\omega \times v)^\varepsilon-\omega^\varepsilon \times v^\varepsilon\|_{L^{\frac{3}{2}}(0,T;L^{\frac{3}{2}}(\Omega))}\|\omega ^\varepsilon\|_{L^3(0,T;L^3(\Omega))}\rightarrow 0,
\end{equation}
due to the condition that $v\in L^3(0,T;B^{\frac{1}{3}}_{3,c(\mathbb{N})}(\Omega))$ and $\omega \in L^3(0,T;L^3(\Omega))$.

(3) It follows from the H\"older inequality that
\be\ba
	|I|\leq &C \|(\omega \times v)^\varepsilon-\omega^\varepsilon \times v^\varepsilon\|_{L^{\frac{p}{p-1}}(0,T;L^{\frac{q}{q-1}})}\|\omega ^\varepsilon\|_{L^p(0,T;L^q)}
\ea\ee
and
$$
\|\omega \times v \|_{L^{\frac{p}{p-1}}(0,T;L^{\frac{q}{q-1}}(\Omega))}\leq C
\|  v \|_{L^{\frac{p}{p-2}}(0,T;L^{\frac{q}{q-2}}(\Omega))}\|\omega\|_{L^p(0,T;L^q(\Omega))}.
$$
Thus, by using the properties of the standard mollifiers,  we know that
$$\limsup_{\varepsilon\rightarrow0}|I|=0.
$$
(4) Recall the identical equation
$$
-\Delta A=\text{curl\,}\text{curl\,} A-\nabla \text{div\,} A.
$$
The classical elliptic estimate yields that
$$\|\nabla A\|_{L^{p}(\Omega)}\leq C(\|\text{curl\,}A\|_{L^{p}(\Omega)}+\|\text{div\,}A\|_{L^{p}(\Omega)}).
$$
Therefore, there holds
$$
\|\nabla v\|_{L^{\f94}(\Omega)}\leq C(\|\text{curl\,}v\|_{L^{\f94}(\Omega)}+\|\text{div\,}v\|_{L^{\f94}(\Omega)}).
$$
When dimemsion $d=3$, the Sobolev inequality implies that
$$\|v\|_{L^{9}(\Omega)}\leq C\|\nabla v\|_{L^{\f94}(\Omega)}.$$
Integrating this estimate in time direction, we derive
$$\|v\|_{L^{3}(0,T;L^{9}(\Omega))}\leq C\|\nabla v\|_{L^{3}(0,T;L^{\f94}(\Omega))}.$$
Finally, we get $v\in L^{3}(0,T;L^{9}(\Omega))$ and $\text{div\,}u,\omega\in L^{3}(0,T;L^{\f{9}{4}}(\Omega))$. From the previous case, we complete the proof of this case.
This concludes Theorem \ref{the01}.
\end{proof}

\section{ Helicity  conservation for the homogeneous incompressible Euler equations}
This section is devoted to the proof of  helicity conservation of weak solutions in   Onsager type spaces $\dot{B}^{\f13}_{p,c(\mathbb{N})}$  for  the  homogeneous incompressible Euler equations.

\begin{proof}[Proof of Theorem \ref{the1.1}]
	Let $v$ be a weak solution to the homogeneous  incompressible Euler equations \eqref{Euler}.  Then one can mollify the equation \eqref{euler} in space direction to deduce that
$$\left\{\begin{aligned}\label{GNS}
&v^{\varepsilon}_{t} +(\omega\times v)^{\varepsilon}+\nabla(\Pi +\f12|v|^{2})^{\varepsilon}= 0,\\
&\omega^{\varepsilon}_{t}-\text{curl}( v\times\omega)^{\varepsilon}=0,\\
& \text{div}\omega^{\varepsilon} =0.
\end{aligned}\right.$$
Consequently, by a straightforward computation, we have
\be\label{3.1}\ba
\f{d}{dt}\int_{\Omega}(v^{\varepsilon}\cdot\omega^{\varepsilon})dx=&\int_{\Omega} v_{t}^{\varepsilon}\cdot
\omega^{\varepsilon}+\omega_{t}^{\varepsilon}\cdot v^{\varepsilon}dx\\
=&\int_{\Omega}-\left((\omega\times v)^{\varepsilon}+\nabla(\Pi+\f12|v|^{2})^{\varepsilon}\right)\cdot \omega^{\varepsilon}+
 \left(\text{curl}\left( v\times\omega\right)\right)^{\varepsilon}\cdot v^{\varepsilon}dx\\
=&\int_{\Omega}-\left(\omega\times v\right)^{\varepsilon}\cdot \omega^{\varepsilon}+
\left(   v\times\omega \right)^{\varepsilon}\cdot\omega^{\varepsilon}dx
 \\
=&\int_{\Omega}-2\Big((\omega\times v)^{\varepsilon} -\omega^{\varepsilon}\times v^{\varepsilon} \Big)\cdot\omega^{\varepsilon}-
 2(\omega^{\varepsilon}\times v^{\varepsilon} )\cdot\omega^{\varepsilon}dx\\
 =&\int_{\Omega}-2\Big((\omega\times v)^{\varepsilon} -\omega^{\varepsilon}\times v^{\varepsilon} \Big)\cdot\omega^{\varepsilon}dx,
 \ea\ee
where  integration by parts and the fact that
 $(   v^{\varepsilon}\times\omega^{\varepsilon} )\cdot\omega^{\varepsilon}=   v^{\varepsilon}\cdot(\omega^{\varepsilon} \times\omega^{\varepsilon})=0$ were utilized.

From  $\eqref{identity}_{1}$ and the divergence-free condition that div$\,v=0$, we infer that
$$ \omega\times v=v\cdot\nabla v-\f{1}{2}\nabla |v|^{2}= \text{div}(v\otimes v)-\f{1}{2}\nabla |v|^{2},$$
which entails that
$$(\omega\times v )^{\varepsilon}=(v\cdot\nabla v)^{\varepsilon}-\f{1}{2}\nabla (|v|^{2})^{\varepsilon}= \text{div}(v\otimes v)^{\varepsilon}-\f{1}{2}\nabla (|v|^{2})^{\varepsilon}$$
and
$$ \omega^{\varepsilon}\times v^{\varepsilon}=v^{\varepsilon}\cdot\nabla v^{\varepsilon}-\f{1}{2}\nabla |v^{\varepsilon}|^{2}= \text{div}(v^{\varepsilon}\otimes v^{\varepsilon})-\f{1}{2}\nabla |v^{\varepsilon}|^{2}.$$
Plugging the latter two equations into \eqref{3.1}, we use integration by parts and observe that
$$\ba
\f{d}{dt}\int_{\Omega} v^{\varepsilon}\cdot\omega^{\varepsilon}dx=&-2\int_{\Omega}\text{div}\big((v\otimes v)^{\varepsilon} -(v^{\varepsilon}\otimes v^{\varepsilon})\big)\cdot \omega^{\varepsilon}+\frac{1}{2}\nabla \left(|v^\varepsilon|^2 -(|v|^2)^\varepsilon\right)\cdot \omega^\varepsilon dx\\
=& 2\int_{\Omega} \big((v\otimes v)^{\varepsilon}-v^{\varepsilon}\otimes v^{\varepsilon}\big) \nabla\omega^{\varepsilon}dx.
    \ea$$
Then integrating the latter relation in time over $(0,t)$, we see that
\be\label{3.2}
\int_{\Omega} v^{\varepsilon}(x,t)\cdot\omega^\varepsilon(x,t)dx- \int_{\Omega} v^{\varepsilon}(x,0)\cdot\omega^\varepsilon(x,0)dx
= 2 \int_{0}^{t}\int_{\Omega} \big(v^{\varepsilon}\otimes v^{\varepsilon}-(v\otimes v)^{\varepsilon}\big)\nabla\omega^{\varepsilon} dxds.
\ee
   Now we are in a position to apply Lemma  \ref{lem2.3} to complete the proof of Theorem \ref{the1.1}.

(1) Note that $ v\in L^{k} (0,T;\dot{B}^{\alpha}_{p,c(\mathbb{N})} )$, we apply Lemma \ref{lem2.3} to obtain
$$
\|(v\otimes v)^{\varepsilon}-v^{\varepsilon}\otimes v^{\varepsilon}\|_{L^{\f{k}{2}}(0,T;L^{\frac{p}{2}}(\Omega))}\leq  o(\varepsilon^{2\alpha}).
$$
Since
 $ \omega\in L^{ \ell  } (0,T; \dot{B}^{\beta}_{q,\infty}),$ we adopt Lemma
  \ref{lem2.2} to get
$$
\|\nabla \omega^{\varepsilon}\|_{L^{\ell}(0,T;L^{q}(\Omega))}\leq O(\varepsilon^{\beta -1}).
$$
In light of the H\"older inequality, we see that
$$\ba
&\left|\int_{0}^{t}\int_{\Omega}\big((v\otimes v)^{\varepsilon}-v^{\varepsilon}\otimes v^{\varepsilon}\big) \nabla\omega^{\varepsilon}dxds\right|\\\leq&C\|(v\otimes v)^{\varepsilon}-v^{\varepsilon}\otimes v^{\varepsilon}\|_{L^{\f{k}{2}}(0,T;L^{\frac{p}{2}}(\Omega))} \|\nabla \omega^{\varepsilon}\|_{L^{\ell}(0,T;L^{q}(\Omega))}\\
\leq& o(\varepsilon^{2\alpha})  O(\varepsilon^{\beta -1})\\
\leq&  o(\varepsilon^{\beta +2\alpha-1}).
\ea$$
Then letting $\varepsilon\rightarrow 0$ in \eqref{3.2}, this completes the proof of the first part.

(2) In the same manner as above, making use of Remark \ref{rem2.3}, we have
$$\ba
&\left|\int_{0}^{t}\int_{\Omega} \big((v\otimes v)^{\varepsilon}-v^{\varepsilon}\otimes v^{\varepsilon}\big) \nabla\omega^{\varepsilon} dxds\right|\\
\leq&C\|(v\otimes v)^{\varepsilon}-v^{\varepsilon}\otimes v^{\varepsilon}\|_{L^{\f{k}{2}}(0,T;L^{\frac{p}{2}}(\Omega))} \|\nabla\omega^{\varepsilon}\|_{L^{\ell}(0,T;L^{q}(\Omega))}\\
\leq& O(\varepsilon^{2\alpha})o(\varepsilon^{\beta-1})\\
\leq&   o(\varepsilon^{\beta +2\alpha-1}).
\ea$$
Then letting $\varepsilon\rightarrow 0$, this together with \eqref{3.2} yields the second part of Theorem \ref{the1.1}.

(3) Choose $q_{1}=q_{2}=\f{3d}{d+2}$. Then there holds $\f{2}{d}+\f{2d-2}{3d}=\f{d+2}{3d}+\f{d+2}{3d}$. Hence, we apply (2) in Lemma \ref{lem2.3} and Remark  \ref{rem2.4} to obtain
$$
\|(v\otimes v)^{\varepsilon}-v^{\varepsilon}\otimes v^{\varepsilon}\|_{L^{\f{3}{2}}(0,T;L^{\f{3d}{2d-2}}(\Omega))}\leq  o(\varepsilon^{\frac{2}{3}}),
$$
 where we have used the condition that
$\omega\in L^{3} (0,T; \dot{B}^{\f13}_{\f{3d}{d+2},c(\mathbb{N})}).$ Following the same path as above, we get
$$\ba
&\left|\int_{0}^{t}\int_{\Omega} \big((v\otimes v)^{\varepsilon}-v^{\varepsilon}\otimes v^{\varepsilon}\big) \nabla\omega^{\varepsilon} dxds\right| \\\leq&C\|(v\otimes v)^{\varepsilon}-v^{\varepsilon}\otimes v^{\varepsilon}\|_{L^{\f{3}{2}}(0,T;L^{\f{3d}{2d-2}}(\Omega))} \|\nabla\omega^{\varepsilon}\|_{L^{3} (0,T; L^{\f{3d}{d+2}}(\Omega))}\\
\leq& O(\varepsilon^{\f23})o(\varepsilon^{-\f23})\\
\leq&  o(1),
\ea$$
which together with \eqref{3.2} concludes the third part of this theorem.

\textbf{(An alternative approach to (3) and \eqref{ccfs}):}\\
Next we provide an alternative approach  to proving this part.
According to the boundedness of Riesz Transform in
homogeneous Besov spaces and the condition that $
 \omega\in L^{3} (0,T; \dot{B}^{\f13}_{\f{3d}{d+2},c(\mathbb{N})}),
$  we have  $
 \nabla v\in L^{3} (0,T; \dot{B}^{\f13}_{\f{3d}{d+2},c(\mathbb{N})}).
$
The Bernstein inequality further helps us to get $
   v\in L^{3} (0,T; \dot{B}^{\f43}_{\f{3d}{d+2},c(\mathbb{N})}).$
    Recall that $\dot{B}^{\f43}_{\f{3d}{d+2},c(\mathbb{N})}\hookrightarrow
    \dot{B}^{\f23}_{3,c(\mathbb{N})}$, we may achieve the desired  proof of (3) from \eqref{ccfs}. \\
    Subsequently, to make the paper more readable and   more self-contained, we also present the proof of \eqref{ccfs}.
 Notice that
$$\|\nabla\omega^{\varepsilon}\|_{L^{q}(\Omega)}\leq C\|\nabla^{2}v^{\varepsilon}\|_{L^{q}(\Omega)}~\text{for}~1<q<\infty,$$
which together with the fact that $ v\in  L^{3}(0,T;B^{\f23}_{3,c(\mathbb{N})}(\mathbb{R}^3))$ and Lemma \ref{lem2.2} means that
$$\|\nabla\omega^{\varepsilon}\|_{L^{3}(\Omega)}\leq C\|\nabla^{2}v^{\varepsilon}\|_{L^{3}(\Omega)}\leq o(\varepsilon^{-\frac{4}{3}}).$$
Using  the H\"older inequality and (1) in Lemma \ref{lem2.3}, we arrive at
$$\ba
&\left|\int_0^t\int_{\Omega} \big((v\otimes v)^{\varepsilon}-v^{\varepsilon}\otimes v^{\varepsilon}\big) \nabla\omega^{\varepsilon} dxds\right|\\\leq&C\|(v\otimes v)^{\varepsilon}-v^{\varepsilon}\otimes v^{\varepsilon}\|_{L^{\f32}(0,T;L^{\f32}(\Omega))} \|\nabla\omega^{\varepsilon}\|_{L^{3}(0,T;L^3(\Omega))}\\
\leq& o(\varepsilon^{\frac{4}{3}})  o(\varepsilon^{-\frac{4}{3}})\\
\leq& o(1).
\ea$$
This together with \eqref{3.2} gives  \eqref{ccfs}.

(4) Following the same path of derivation of  (2) in Theorem \ref{the1.1}, Lemma \ref{lem2.7} gives the proof of this part.

(5) With \eqref{3.2} and  div$\,\omega=0$ in hand, a slight variant of the   proof of  (2) in Theorem \ref{the1.1} provides the proof of this part.
\end{proof}
\section{  Helicity  conservation for 2-D surface quasi-geostrophic equations}
In this section, we are concerned with  the helicity  conservation of weak solutions to 2-D surface quasi-geostrophic equations.
\begin{proof}[Proof of Theorem \ref{the1.2}]
First, due to the divergence-free condition of velocity $v(x,t)$, we rewrite the equation \eqref{qg} as
\be\left\{\ba\label{qg1}
&\theta_{t} +\Div( v\otimes
\theta)=0,\\
&v(x,t)=(-\mathcal{R}_2 \theta, \mathcal{R}_1 \theta ),\\
&\theta|_{t=0}=\theta_0,
\ea\right.\ee
where $\mathcal{R}_i$ represents the Riesz transform for $i=1,2$.
Next, we regularize the equation \eqref{qg1} to obtain
$$\theta^{\varepsilon}_{t}+\partial_{j} (v_{j}\theta)^{\varepsilon}=0,$$
and
$$\partial_{i}\theta^{\varepsilon}_{t}+\partial_{j} (\partial_{i}v_{j}\theta)^{\varepsilon}+\partial_{j} (v_{j}\partial_{i}\theta)^{\varepsilon}=0.$$
Thus, it follows from a direct computation that for $i=1,2$,
\be\label{3.3}\ba
\f{d}{dt}\int_{\Omega}\theta^{\varepsilon}\partial_{i}\theta^{\varepsilon}dx
&=\int_{\Omega}\theta^{\varepsilon}\partial_{i}\partial_{t}\theta^{\varepsilon}dx
+\int_{\Omega}\partial_{t}\theta^{\varepsilon}\partial_{i}\theta^{\varepsilon}dx\\
&=-\int_{\Omega}\theta^{\varepsilon}[\partial_{j} (\partial_{i}v_{j}\theta)^{\varepsilon}+\partial_{j} (v_{j}\partial_{i}\theta)^{\varepsilon}]dx-\int_{\Omega} \partial_{j} (v_{j}\theta)^{\varepsilon}\partial_{i}\theta^{\varepsilon}dx.
\ea\ee
Using the divergence-free condition of $v(x,t)$ again, we have the fact that $\int \partial_{j} (\partial_{i}v_{j}^{\varepsilon}\theta^{\varepsilon})\theta^{\varepsilon}dx=0$. Then we can reformulate \eqref{3.3} as
$$\ba
\f{d}{dt}\int_{\Omega}\theta^{\varepsilon}\partial_{i}\theta^{\varepsilon}dx
=&-\int_{\Omega}\theta^{\varepsilon} \big[\partial_{j} (\partial_{i}v_{j}\theta)^{\varepsilon}-\partial_{j} (\partial_{i}v_{j}^{\varepsilon}\theta^{\varepsilon})\big] dx-\int_{\Omega}\theta^{\varepsilon}\big[\partial_{j} (v_{j}\partial_{i}\theta)^{\varepsilon} -\partial_{j} (v_{j}^{\varepsilon}\partial_{i}\theta^{\varepsilon})\big] dx\\&-\int_{\Omega}\theta^{\varepsilon}\partial_{j} (v_{j}^{\varepsilon}\partial_{i}\theta^{\varepsilon}) dx  -\int_{\Omega} \partial_{j} (v_{j}\theta)^{\varepsilon}\partial_{i}\theta^{\varepsilon}dx\\
=&\int_{\Omega}\partial_{j} \theta^{\varepsilon}\big [ (\partial_{i}v_{j}\theta)^{\varepsilon}- (\partial_{i}v_{j}^{\varepsilon}\theta^{\varepsilon})\big] dx+\int_{\Omega}\big[ (v_{j}\partial_{i}\theta)^{\varepsilon} - (v_{j}^{\varepsilon}\partial_{i}\theta^{\varepsilon})\big] \partial_{j}\theta^{\varepsilon} dx\\&-\int_{\Omega}\theta^{\varepsilon} v_{j}^{\varepsilon}\partial_{i}\partial_{j}\theta^{\varepsilon}  dx  +\int_{\Omega} (v_{j}\theta)^{\varepsilon}\partial_{j}\partial_{i}\theta^{\varepsilon}dx\\
=&\int_{\Omega}\partial_{j} \theta^{\varepsilon} \big[ (\partial_{i}v_{j}\theta)^{\varepsilon}- (\partial_{i}v_{j}^{\varepsilon}\theta^{\varepsilon})\big] dx+\int_{\Omega}\big[ (v_{j}\partial_{i}\theta)^{\varepsilon} - (v_{j}^{\varepsilon}\partial_{i}\theta^{\varepsilon})\big] \partial_{j}\theta^{\varepsilon} dx\\&+\int_{\Omega}\big[(v_{j}\theta)^{\varepsilon}-\theta^{\varepsilon} v_{j}^{\varepsilon}\big]\partial_{i}\partial_{j}\theta^{\varepsilon}  dx,
\ea$$
which leads to
$$\ba
&\int_{\Omega}\theta^{\varepsilon}(x,t)\partial_{i}\theta^{\varepsilon}(x,t)dx
-\int_{\Omega}\theta^{\varepsilon}(x,0)\partial_{i}\theta^{\varepsilon}(x,0)dx
\\
= &\int_{0}^{t}\int_{\Omega}\partial_{j} \theta^{\varepsilon}\big [ (\partial_{i}v_{j}\theta)^{\varepsilon}- (\partial_{i}v_{j}^{\varepsilon}\theta^{\varepsilon})\big] dxds+\int_{0}^{t}\int_{\Omega}\big[ (v_{j}\partial_{i}\theta)^{\varepsilon} - (v_{j}^{\varepsilon}\partial_{i}\theta^{\varepsilon})\big] \partial_{j}\theta^{\varepsilon} dxds\\&+\int_{0}^{t}\int_{\Omega}\big[(v_{j}\theta)^{\varepsilon}-\theta^{\varepsilon} v_{j}^{\varepsilon}\big]\partial_{i}\partial_{j}\theta^{\varepsilon}  dxds  \\
=:&
I+II+III.
\ea$$
For  the first term $I$, we conclude by the H\"older inequality that
$$
|I|\leq\|(\partial_{i}v_{j}\theta)^{\varepsilon}- (\partial_{i}v_{j}^{\varepsilon}\theta^{\varepsilon}) \|_{L^{\f32}(0,T;L^{\f65}(\Omega))} \| \partial_{j} \theta^{\varepsilon}\|_{L^{3}(0,T;L^{6}(\Omega))}.
$$
Take
$q=\f65,d=2, q_{1}=\f32, q_{2}=\f32<2.$ Note that $\f12+\f56=\frac{1}{q_{1}}+\frac{1}{q_{2}}$ and
 $\nabla\theta\in L^{3}(0,T;\dot{B}^{\f13}_{ \f32, c(\mathbb{N})})$, we can invoke (3) in Lemma \ref{lem2.3} to get
$$\|(\partial_{i}v_{j}\theta)^{\varepsilon}- (\partial_{i}v_{j}^{\varepsilon}\theta^{\varepsilon}) \|_{L^{\f32}(0,T;L^{\f65}(\Omega))}
\leq o(\varepsilon^{\frac{2}{3}}).$$
Combining
Sobolev's inequality and  Lemma \ref{lem2.2}, we deduce from $\nabla\theta\in L^{3}(0,T;\dot{B}^{\f13}_{ \f32, c(\mathbb{N})}) $ that
$$ \| \partial_{j} \theta^{\varepsilon}\|_{L^{3}(0,T;L^{6}(\Omega))}\leq C\|\nabla \partial_{j} \theta^{\varepsilon}\|_{L^{3}(0,T;L^{\f32}(\Omega))}\leq o(\varepsilon^{-\frac{2}{3}} ).$$
Collecting the above estimates, we get
$$
|I|\leq o(1 ).
$$
Likewise, we have
$$
|II|\leq o(1).
$$
It remains to estimate the third term $III$. Using the H\"older inequality, we find
$$
|III|\leq  \|(v_{j}\theta)^{\varepsilon}-\theta^{\varepsilon} v_{j}^{\varepsilon} \|_{L^{\f32}(0,T;L^{3}(\Omega))} \| \partial_{i}\partial_{j}\theta^{\varepsilon}\|_{L^{3}(0,T;L^{\f32}(\Omega))}.
$$
Choosing
$q=3,d=2, q_{1}=\f32=q_{2}=\f32<2,$  (2) in Lemma \ref{lem2.3} enables us
 to get
$$ \|(v_{j}\theta)^{\varepsilon}-\theta^{\varepsilon} v_{j}^{\varepsilon} \|_{L^{\f32}(0,T;L^{3}(\Omega))}\leq o(\varepsilon^{\frac{2}{3} } ),$$
where we have used the condition that $\nabla\theta\in L^{3}(0,T;\dot{B}^{\f13}_{ \f32, c(\mathbb{N})})$.
Besides, Lemma \ref{lem2.2} allows us to conclude that
$$\| \partial_{i}\partial_{j}\theta^{\varepsilon}\|_{L^{3}(0,T;L^{\f32}(\Omega))}\leq o(\varepsilon^{-\frac{2}{3}} ).$$
In summary, we know that
$$|III|\leq o(1 ).$$
We achieve the proof of this theorem.
\end{proof}

\section*{Acknowledgements}

 Wang was partially supported by  the National Natural
 Science Foundation of China under grant (No. 11971446, No. 12071113   and  No.  11601492).
 Wei was partially supported by the National Natural Science Foundation of China under grant (No. 11601423, No. 11871057).  Ye was partially supported by the National Natural Science Foundation of China  under grant (No.11701145) and China Postdoctoral Science Foundation (No. 2020M672196).

\end{document}